\documentclass[11pt, reqno ]{amsart}
\usepackage{geometry}                
\usepackage{graphicx}
\usepackage{amssymb}
\usepackage{epstopdf}
\usepackage{amsmath}
\usepackage{mathrsfs}
\usepackage{tikz-cd}
\usepackage[all]{xy}
\usepackage{color}
\usepackage{tikz}
\usepackage{xcolor}
\usepackage{mathtools}
  \usepackage{enumitem}
  \usepackage[all]{xy}
\usepackage{color}
\usepackage{tikz}
  \usepackage{array}
  \usepackage{amsthm}

\definecolor{red}{RGB}{200,26,26} 
\definecolor{blue}{RGB}{0,0,255}  
\definecolor{green}{RGB}{28,133,26} 
\definecolor{mauve}{RGB}{100,10,100} 
  
\usepackage[colorlinks=true, pdfstartview=FitV,linkcolor=blue,citecolor=blue,urlcolor=blue]{hyperref}

    \advance \topmargin by -\headheight
    \advance \topmargin by -\headsep

    \evensidemargin \oddsidemargin
\numberwithin{equation}{section}

\def\cC{{\mathcal C}}

\def\gg{{\mathfrak g}}

\newcommand{\vac}{\mathbf1}
\newfont{\german}{eufm10}

\begin{document}

\newtheorem{maintheorem}{Main Theorem}
\newtheorem{thm}[subsection]{Theorem}
 \newtheorem{theorem}{Theorem}[section]
  \newtheorem{prop}[theorem]{Proposition}
  \newtheorem{cor}[theorem]{Corollary}
  \newtheorem{lemma}[theorem]{Lemma}
  \newtheorem{defn}[theorem]{Definition}
  \newtheorem{ex}[theorem]{Example}
    \newtheorem{conj}[theorem]{Conjecture}
     \newtheorem{remark}[theorem]{Remark}
          \newtheorem{example}[theorem]{Example}

 \newcommand{\cx}{{\bf C}}
\newcommand{\la}{\langle}
\newcommand{\ra}{\rangle}
\newcommand{\res}{{\rm Res}}
\newcommand{\expp}{{\rm exp}}
\newcommand{\lgp}{\widehat{\rm G}( { F} ((t)) )}
\newcommand{\svee}{\scriptsize \vee}
\newcommand{\deff}{\stackrel{\rm def}{=}}
\newcommand{\cal}{\mathcal}
\newcommand{\ch}{\operatorname{ch}}
\newcommand{\on}{\operatorname}
\newcommand{\mc}{\mathcal}
\newcommand{\Bcat}{\mathscr B}
\newcommand{\Ext}{\operatorname{Ext}}
\newcommand{\smod}{\textup{-mod}}
\newcommand{\Irr}{\operatorname{Irr}}
\newcommand{\rad}{\operatorname{rad}}

\title{Feigin-Semikhatov duality at the critical level}

\author{Thomas Creutzig $^{1}$}

\email{thomas.creutzig@fau.de}

\author{Xuanzhong Dai $^{2}$}
\email{xzdai@nju.edu.cn}

\author{Bailin Song $^{3}$}
\email{bailinso@ustc.edu.cn}

\address{$^1$ Department of Mathematics, FAU Erlangen-Nürnberg, Erlangen, Germany}
\address{$^2$ School of Mathematics, Nanjing University, Nanjing 210093, China}
\address{$^3$ School of Mathematical Sciences, University of Science and Technology of China, Hefei, Anhui 230026, CHINA}

\begin{abstract}
The Feigin-Semikhatov duality asserts that the Heisenberg cosets of the subregular $W$-algebra of $\mathfrak{sl}_n$ at level $k$ and the one of the principal $W$-superalgebra of $\mathfrak{sl}_{n|1}$ at level $\ell$ coincide when the levels satisfy the Feigin-Frenkel relation $(k+n)(\ell+n-1)=1$. 
A similar duality holds between the subregular $W$-algebra of $\mathfrak{so}_{2n+1}$ and the principal $W$-superalgebra of $\mathfrak{osp}_{2|2n}$.

We study these dualities in the critical/large level limit. 
We describe the centerless subregular $W$-algebra at the critical level as an orbifold of the large level limit of the principal $W$-superalgebra times a lattice VOA. 
Our construction yields a functor between certain categories of the two involved vertex algebras. We show that in this set-up one in fact gets block-wise equivalences of categories. Studying the principal block of the large level limit of the principal $W$-superalgebra then gives us the structure of the principal blocks of the subregular $W$-algebras in the category of weight modules (which is much larger than the more common category of lower bounded modules).
\end{abstract}

\maketitle

\section{Introduction}

Let $G$ be a connected complex simple algebraic group whose Lie algebra is $\mathfrak g$,
we normalize the Killing form in the usual way so that long roots have norm two. To any complex number $k$ one can then construct the affine vertex algebra $V^k(\mathfrak g)$ of $\mathfrak g$ at level $k$. Let $f$ be a nilpotent element in $\mathfrak g$, then one can further get another vertex algebra via quantum Hamiltonian reduction, $W^k(\mathfrak g, f)$, the $W$-algebra of $\mathfrak g$ associated to $f$ at level $k$. It only depends on the nilpotent orbit of $f$. 
Let $h^\vee$ be the dual Coxeter number of $\mathfrak g$. It then turns out that the center of $V^k(\mathfrak g)$ is trivial unless $k$ is at the critical level $-h^\vee$. In this case the center coincides with $W^{-h^\vee}(\mathfrak g, f_{\text{prin}})$, the principal $W$-algebra at the critical level. 
Principal $W$-algebras satisfy Feigin-Frenkel duality \cite{FF2}, that is
\[
W^k(\mathfrak{g}, f_{\text{prin}}) \cong W^\ell({}^L\mathfrak{g}, f_{\text{prin}})
\]
for 
\[
r (k+h^\vee)(\ell + {}^Lh^\vee) =1
\]
and ${}^L\mathfrak{g}$ the Langlands dual of $\mathfrak g$, that is the Lie algebra whose root system is the coroot system of $\mathfrak g$. Here $r$ is the lacing number of $\mathfrak g$. At the critical level the dual level goes to infinity and this correspondence becomes the statement that the center $W^{-h^\vee}(\mathfrak{g}, f_{\text{prin}})$ is isomorphic to Fun(Op$_{{}^LG}$), the algebra of functions on the moduli space of ${}^LG$-opers on the formal disc, see Theorem 4.3.1 of \cite{Frenkel-book}. This Theorem is called Langlands correspondence for loop groups and the identification of $V^{-h^\vee}(\mathfrak g)$-mod with the category of quasi-coherent sheaves on Op$_{{}^LG}$ is a major aspect of this Langlands correspondence. 
It should be viewed as a loop group analogue of the Beilinson-Bernstein Theorem saying that taking global sections gives an equivalence between $D$-modules on the flag variety of $G$ and the block of $\mathfrak g$ with trivial central character \cite{BB}.

The duality of Feigin and Frenkel has an analogue for many nilpotent elements in type A \cite{CL1, CKLSS} 
and types B, C, D \cite{CL2}. However the role of the dual Lie algebra is then played by a Lie superalgebra. The type A case of \cite{CL1} is as follows. 
Let $n, m  \in \mathbb Z_{>0}$ and $\mathfrak{g} = \mathfrak{sl}_{n+m}$. Nilpotent orbits in $\mathfrak{sl}_{n+m}$ are parameterized by partitions of the integer $n+m$ and we set $f_{n, m}$ the nilpotent element whose orbit corresponds to the partition $(n, 1, \dots, 1)$. The $W$-algebra $W^k(\mathfrak{sl}_{n+m}, f_{n, m})$ has $V^{k+n-1}(\mathfrak{gl}_{m})$ as subalgebra. 
On the dual side we consider $\mathfrak{g} = \mathfrak{sl}_{n+m|m}$ and the nilpotent element $f_{n+m|m}$ that corresponds to the pair of partitions $(n+m|1, \dots, 1)$ of the pair of integers $(n+m|m)$.
The $W$-algebra $W^\ell(\mathfrak{sl}_{n+m|m}, f_{n+m|m})$ has $V^{-(\ell+n+m-1)}(\mathfrak{gl}_{m})$ as subalgebra and the analogue of Feigin-Frenkel duality says that for generic levels
\[
\text{Com}\left(V^{k+n-1}(\mathfrak{gl}_{m}),  W^k(\mathfrak{sl}_{n+m}, f_{n, m})\right) \cong
\text{Com}\left(V^{-(\ell+n+m-1)}(\mathfrak{gl}_{m}), W^\ell(\mathfrak{sl}_{n+m|m}, f_{n+m|m})\right) 
\]
where the levels satisfy the Feigin-Frenkel type duality relation
\[
(k+n+m)(\ell +n) = 1. 
\]
 Note that the dual Coxeter numbers of $\mathfrak{sl}_{n+m}$ and $\mathfrak{sl}_{n+m|m}$ are $n+m$ and $n$ respectively. This isomorphism is true as one-parameter families of vertex algebras, i.e. it holds for generic levels. 
This isomorphism had appeared in the physics literature in the case $n=m=1$ and there it firmates under the name Kazama-Suzuki duality \cite{KS}. 
This duality alone is already significant, e.g. it was key to prove that the category of weight modules of the affine vertex algebra of $\mathfrak{sl}_2$ at any admissible level has a vertex tensor category structure \cite{C}.
For general $n$, but still $m=1$ this has been a conjecture by Feigin and Semikhatov \cite{FS} and it had been first proven in \cite{CGN}.
 
In type $B$ there is an analogous Feigin-Frenkel type duality. Consider $\mathfrak g = \mathfrak{so}_{2m+2n+1}$ and set $f_{n, m}$ the nilpotent element whose orbit corresponds to the partition $(2n+1,1, \dots, 1)$ of $2m+2n+1$. Then $W^k(\mathfrak{so}_{2m+2n+1}, f_{n, m})$ has $V^{k+2n}(\mathfrak{so}_{2m})$ as subalgebra. On the dual side consider $\mathfrak g = \mathfrak{osp}_{2m|2(n+m)}$ and the nilpotent element $f_{n+m|m}$ corresponding to the partition $(1, \dots, 1|2m+2n)$ of $(2m|2m+2n)$. Then $W^\ell( \mathfrak{osp}_{2m|2(n+m)}, f_{n+m|m})$ has $V^{-2(\ell+m+n)+1}(\mathfrak{so}_{2m})$ as subalgebra and the Feigin-Frenkel type duality \cite{CL2} is
\[
\text{Com}\left(V^{k+2n}(\mathfrak{so}_{2m}),  W^k(\mathfrak{so}_{2(m+n)+1}, f_{n,m})\right) \cong
\text{Com}\left(V^{-2(\ell+ m+n)+1}(\mathfrak{so}_{2m}), W^\ell( \mathfrak{osp}_{2m|2(n+m)}, f_{n+m|m})\right) 
\]
where the levels satisfy the Feigin-Frenkel type duality relation
\[
2(k+2n+2m-1)(\ell +n+1) = 1. 
\]

These are thus not  isomorphisms of $W$-algebras, but of their coset subalgebras. Nonetheless there are convolution operations between these $W$-algebras that provide quite useful functors between representation categories \cite{CLNS}. 
The convolution kernels are certain shifted versions of the vertex algebra of chiral differential operators. Their existences has been conjectured in \cite{CG} and these algebras have been named quantum geometric Langlands kernel VOAs. This conjecture has been proven in few low  rank cases in \cite{CG, CGL} and then in generality for types A, D by Yuto Moriwaki \cite{Mo}. For $\mathfrak g = \mathfrak{gl}_m$ it has the following form
\begin{equation*}
    \begin{split}
\mathfrak A^n[\mathfrak{gl}_m, k] &= \bigoplus_{\lambda \in P^+} V^k_\lambda \otimes V^\ell_{\lambda^*} \otimes V_{\sqrt{nm}\mathbb Z + \frac{n\omega(\lambda)}{\sqrt{nm}}} \otimes \pi \\
\mathfrak A^n[\mathfrak{so}_{2m}, k] &= \bigoplus_{\lambda \in R^+} V^k_\lambda \otimes V^\ell_{\lambda^*}..
    \end{split}
\end{equation*}
Here $P^+$ is the set of dominant weights of $\mathfrak{sl}_m$ and $R^+$ is the set of dominant weights of $\mathfrak{so}_{2m}$ that are non-spin (they lie in the lattice $Q \cup (Q+\omega_1)$ with $Q$ the root lattice of $\mathfrak{so}_{2m}$ and $\omega_1$ the first fundamental weight).
$V^k_\lambda$ is the Weyl module of $\widehat{\mathfrak{sl}}_m$ of highest-weight $\lambda$ at level $k$. $\lambda^*$ is the weight dual to $\lambda$, $\pi$ is a simple rank one Heisenberg vertex algebra and the map $\omega: P \rightarrow \mathbb Z/m\mathbb Z$ satisfies $\omega(\lambda) = i$ if $\lambda = \omega_i \mod Q$ with $\omega_i$ the $i$-the fundamental weight, the convention that $\omega_0 = 0$ and $Q$ the root lattice of 
$\mathfrak{sl}_m$. Finally the upper index $n$ tells us that the level $\ell$ satisfies
\[
\frac{1}{k+h^\vee} + \frac{1}{\ell+h^\vee} = n
\]
with $h^\vee = m$, resp. $2m-2$ the dual Coxeter number of $\mathfrak{sl}_m$, resp $\mathfrak{so}_{2m}$. 
In the case of $m=1$,  this is simply $\mathfrak A^n[\mathfrak{gl}_1, k] \cong V_{\sqrt{n}\mathbb Z} \otimes \pi \cong \mathfrak A^n[\mathfrak{so}_{2}, k]$.
Set $\tilde k = -k - n -2m +1$ and $\tilde \ell = \ell + n -m -1$ in the case of type $A$ and $\tilde k = -k - 2n -4m +2$ and $\tilde \ell = 2(\ell + n  -m +1) +1$ in the case of type $B$. Then the convolution of $W$-algebra statement for generic levels and for our cases is \cite{CLNS}
\begin{equation}\nonumber
\begin{split}
    H^{\text{rel}}(\mathfrak{gl}_m, \mathfrak A^1[\mathfrak{gl}_m, \tilde k] \otimes W^k(\mathfrak{sl}_{n+m}, f_{n, m})) &\cong W^\ell(\mathfrak{sl}_{n+m|m}, f_{n+m|m}) \\
     H^{\text{rel}}(\mathfrak{gl}_m, \mathfrak A^{-1}[\mathfrak{gl}_m, \tilde \ell] \otimes W^\ell(\mathfrak{sl}_{n+m|m}, f_{n+m|m})) &\cong W^k(\mathfrak{sl}_{n+m}, f_{n, m})\\
     H^{\text{rel}}(\mathfrak{so}_{2m}, \mathfrak A^1[\mathfrak{so}_{2m}, \tilde k] \otimes W^k(\mathfrak{so}_{2(n+m)+1}, f_{n, m})) &\cong W^\ell(\mathfrak{osp}_{2m|2(n+m)}, f_{n+m|m}) \\
     H^{\text{rel}}(\mathfrak{so}_{2m}, \mathfrak A^{-1}[\mathfrak{so}_{2m}, \tilde \ell] \otimes W^\ell(\mathfrak{osp}_{2m|2(n+m)}, f_{n+m|m})) &\cong W^k(\mathfrak{so}_{2(n+m)+1}, f_{n, m}).
\end{split}   
\end{equation}
Here $H^{\text{rel}}$ denotes the relative semi-infinite Lie algebra cohomology of \cite{FGZ} that we will discuss in section \ref{sec:BRST} in the instance of $\mathfrak g = \mathfrak{gl}_1$. For more details on this see \cite{CLNS}.
Replacing the $W$-algebras by their modules in the cohomology leads to a functor between representation categories. In particular in the case of $\mathfrak g = \mathfrak{gl}_1$ it yields block-wise equivalences of categories and the functor also respects the associativity of intertwining operators, see \cite{CGNS}.

The aim of this work is to generalize the Feigin-Semikhatov duality to the critical level and hence the dual level goes to $\infty$. We now present the general philosophy. 
The first idea is that one can take the large level limit of $\mathfrak A^n[\mathfrak{g}, k]$ in such a way that the limit admits a large center and that the quotient is a simple vertex algebra, that decomposes as
\begin{equation*}
    \begin{split}
        \mathrm{FT}[\mathfrak{gl}_m, n] &= \bigoplus_{\lambda \in P^+} \rho_\lambda \otimes V^{-m + \frac{1}{n}}_{\lambda^*} \otimes V_{{\sqrt{nm}\mathbb Z + \frac{n\omega(\lambda)}{\sqrt{mn}}}} \\
        \mathrm{FT}[\mathfrak{so}_{2m}, n] &= \bigoplus_{\lambda \in R^+} \rho_\lambda \otimes V^{-2m +2 + \frac{1}{n}}_{\lambda^*} .
    \end{split}
\end{equation*}
Here $\rho_\lambda$ is the integrable representation of highest-weight $\lambda$ of $\mathfrak{sl}_m$ resp $\mathfrak{so}_{2m}$ and FT stands for (affine) Feigin-Tipunin (algebra) as they first suggested principal $W$-algebra analogues of these algebras \cite{FT}. These types of algebras have obtained recent attention \cite{Su, C2, CDGG, FL1, FL2, FL3}. 
For $m=1$ the affine Feigin Tipunin algebra is just a lattice VOA, $\mathrm{FT}[\mathfrak{gl}_1, n] \cong V_{\sqrt{n}\mathbb Z} \cong \mathrm{FT}[\mathfrak{so}_{2}, n]$. For $n=-1$ it should be closely related to the chiral Hecke algebra \cite{FB}, however the latter is to our knowledge not known to be simple. Finally for $\mathfrak{sl}_2$ and $n$ positive these algebras have been constructed \cite{CNS} and then later also in \cite{ACGY}.
In particular $\mathrm{FT}[\mathfrak{gl}_2, 1] \cong L_1(\mathfrak{psl}_{2|2})$.
For the construction and proof of simplicity of $\mathrm{FT}[\mathfrak{gl}_2, -1]$ the work \cite{AM} is probably quite useful. 

Let $(\mathfrak a, \mathfrak{g}_1, \mathfrak{g}_2, G) \in \{ (\mathfrak{gl}_m, \mathfrak{sl}_{n+m}, \mathfrak{sl}_{n+m|m}, GL(m)), (\mathfrak{so}_{2m}, \mathfrak{so}_{2(n+m)+1}, \mathfrak{osp}_{2m|2(n+m)}, SO(2m)) \}$ and let $h^\vee_1, h^\vee_2$ be the dual Coxeter numbers of $\mathfrak{g}_1, \mathfrak{g}_2$. 
We then conjecture that at the critical level/large level limit 
the dualities take the form
\begin{align}
(W_\infty(\mathfrak{g}_{1}, f_{n,m})
 \otimes \mathrm{FT}[\mathfrak{a},1])^{G}
&\cong
W_{-h^\vee_2}(\mathfrak{g}_{2}, f_{n+m|m}),
\\
H^{\mathrm{rel}}\!\left(
\mathfrak{a},\,
\mathrm{FT}[\mathfrak{a},-1]
\otimes
W_{-h^\vee_2}(\mathfrak{g}_{2}, f_{n+m|m})
\right)
&\cong
W_\infty(\mathfrak{g}_{1}, f_{n,m}),
\\
(W_\infty(\mathfrak{g}_{2}, f_{n+m|m})
 \otimes \mathrm{FT}[\mathfrak{a},-1])^{G}
&\cong
W_{-h^\vee_1}(\mathfrak{g}_{1}, f_{n,m}),
\label{eq:conjdual1}
\\
H^{\mathrm{rel}}\!\left(
\mathfrak{a},\,
\mathrm{FT}[\mathfrak{a},1]
\otimes
W_{-h^\vee_1}(\mathfrak{g}_{1}, f_{n,m})
\right)
&\cong
W_\infty(\mathfrak{g}_{2}, f_{n+m|m}).
\label{eq:conjdual2}
\end{align}
where $W_{-h^\vee_1}(\mathfrak{g}_{1}, f_{n, m}), W_{-h^\vee_2}(\mathfrak{g}_{2}, f_{n+m|m})$ denote the centerless critical level $W$-algebras.  Moreover $W_\infty(\mathfrak{g}_{1}, f_{n, m}), W_\infty(\mathfrak{g}_{2}, f_{n+m|m}) $ are certain large level limits, see Section \ref{sec:limits} for those relevant for us. These algebras are (generalized) free field algebras, see \cite{CL1, CL2} and our Definition \ref{defn:freefield}. 
The representation theory of the free field algebras is non semi-simple and in particular the principal block can be well studied. It is natural to expect that this duality gives in particular interesting information about the abelian category of the critical level $W$-algebras. 

\begin{remark} \label{rmk:2limits}
There are different choices on taking the large level limit, Section \ref{sec:limits}. If $\gg$ is not a Lie algebra but only a Lie superalgebra then it seems as there are also at least two choices of defining the critical level $W$-superalgebra. Firstly as the quantum Hamiltonian reduction of the affine vertex superalgebra and secondly by viewing the W-superalgebra as a one-parameter family of vertex superalgebras (with the parameter being the level) obtained from the one-parameter family of affine vertex superalgebras via quantum Hamiltonian reduction and then choosing a specialization to the critical level. These two procedures need not to coincide and for example for the $W$-superalgebra of $\mathfrak{sl}_{2|1}$ the first procedure leads to a vertex superalgebra without Virasoro field and with a center, while the maybe most natural specialization in the second case leads to a simple vertex operator superalgebra. 

The conjecture should work for either case as long as one takes the appropriate large level limit on the dual side. 
Drazen Adamovic, Shigenori Nakatsuka and Boris Feigin already related the centers of critical level principal $W$-superalgebras of $\mathfrak{gl}_{n|1}$ to commutative algebras associated to large level limits of the dual side \cite{ AFN, AN}. 
\end{remark}

\subsection{Feigin-Semikhatov duality at the critical level}

We focus on the subregular case, $m=1$. In this case the following statements are known in type $A$: the OPE algebra is known \cite{GK}; the critical level subregular $W$-algebra is  realized as a coset inside $n^2$ $\beta\gamma$-VOAs \cite{CGL1, CGL2}, in particular using Zhu's algebra (which is a Smith algebra) simple lower-bounded modules have been classified in \cite{CGL2}.

One of our main results is to show the duality  in the critical level /large level limit (\ref{eq:conjdual1}) and (\ref{eq:conjdual2}) for $m=1$. 
\begin{theorem} \label{thm:KScritical}
Let $(\mathfrak g_1,\mathfrak g_2)=(\mathfrak {sl}_{n+1},\mathfrak {sl}_{n+1|1})$ or $(\mathfrak{so}_{2n+1},\mathfrak {osp}_{2|2n})$, and let $h_1^\vee$ be the dual Coxeter number of $\mathfrak g_1$ and $f_{\mathrm{sub}}$  be subregular nilpotent element of $\mathfrak g_1$. 
The Kazama-Suzuki type isomorphisms hold for the large-critical level duality:
\begin{align*} 
(W_{\infty}(\mathfrak g_2)\otimes V_{\sqrt{-1}\mathbb Z} )^{\mathbb C^*} & \cong W_{-h_1^\vee}(\mathfrak g_1,f_{\mathrm{sub}}),\\
H^{\mathrm{rel},0}\!\left( \mathfrak{gl}_1,
  W_{-h_1^\vee}(\mathfrak g_1, f_{\mathrm{sub}})\otimes V_{\mathbb Z}
\right)
&\cong  W_{\infty}(\mathfrak g_2).
\end{align*}
\end{theorem}

As a corollary of Theorem \ref{thm:KScritical}, we prove that
 the two large level limit $W$-algebras $W_{\infty}(\mathfrak{sl}_{2n|1}) $ and $W_{\infty}(\mathfrak{osp}_{2|2n})$ are isomorphic and hence
\begin{cor}
  There is an isomorphism of vertex algebras
\[
W_{-2n}(\mathfrak{sl}_{2n},f_{\mathrm{sub}})
\cong
W_{-2n+1}(\mathfrak{so}_{2n+1},f_{\mathrm{sub}}).
\]  
\end{cor}
Andrew Linshaw made us aware of the possibility of such isomorphisms and in fact he explained to us that it is even believed that there is a surjection  of universal critical level subregular $W$-algebras, $W^{-2n}(\mathfrak{sl}_{2n},f_{\mathrm{sub}})
\twoheadrightarrow
W^{-2n+1}(\mathfrak{so}_{2n+1},f_{\mathrm{sub}}).$

\subsubsection{Consequences for representation theory}

The subregular $W$-algebras contain a Heisenberg vertex subalgebra
and admit a Hamiltonian $H$ even at the critical level, where
the standard conformal vector is unavailable.
In analogy with \cite[Definition~3.1]{ACK}, we introduce the
following categories of modules with semisimple Heisenberg
zero mode admitting a compatible Hamiltonian, subject to
the conditions in Definition~\ref{defcatwtmod}.

Let $V$ be a  vertex algebra with a Hamiltonian and a Heisenberg element. Then we define
\begin{enumerate}
\item $V$-\textup{wtmod}:
 the category of finitely generated $V$-modules with finite-dimensional weight spaces, that is an object $M$ is finitely generated and graded by conformal weight and Heisenberg weight, that is 
\[
M = \bigoplus_{\lambda, \Delta} M_{\lambda, \Delta}
\]
and
$\text{dim}\   M_{\lambda, \Delta} < \infty$ for any $(\lambda, \Delta)$ and for each $\lambda$ there exists $h_\lambda$, such that $M_{\lambda, \Delta} = 0$ for $\text{Re}(\Delta) < \text{Re}(h_\lambda)$.
\item $V$-\textup{wtmod}$_{\geq}$: the full subcategory
of $V$-\textup{wtmod} consisting of objects that are 
finitely generated lower bounded weight modules, that is $M$ in $V$-\textup{wtmod}$_{\geq}$, if it is in $V$-\textup{wtmod} and if Hamiltonian weight is lower bounded. 

\item $V$-\textup{wtmod}$_{\mathcal O}$: the full subcategory
of $V$-\textup{wtmod}$_{\geq}$: consisting of objects such that for every weight $\Delta$ there exists an $N \in \mathbb Z$, such that $M_{\lambda, \Delta} = 0$ for $\text{Re}(\lambda) \geq N$.

\item $V$-\textup{wtmod}$_{KL}$: the full subcategory
of $V$-\textup{wtmod}$_{\mathcal O}$ consisting of objects  with finite-dimensional Hamiltonian weight spaces.

\end{enumerate}
If $V$ is a vertex operator algebra, that is, it has a conformal vector, then  one takes the Hamiltonian as the Virasoro zero-mode and recovers \cite[Definition~3.1]{ACK}  for $V$ an affine vertex algebra at non-critical level. 

The only case where modules of the category of weight modules are classified is the case of the simple affine vertex algebra  of $\mathfrak{sl}_2$, $V_k(\mathfrak{sl}_2)$, at any admissible level \cite{ACK}, building on many earlier works \cite{AM2, KR2, A}. Subsequently it was then possible to show that this category is a ribbon category \cite{Flor, C3, C, CMY}.
The proofs of the existence of vertex tensor category \cite{C} as well as the computation of fusion rules of \cite{Flor} relied on the Kazama-Suzuki/Feigin-Semikhatov duality with the $N=2$ super conformal algebra. Note that the latter is the principal $W$-superalgebra of $\mathfrak{sl}_{2|1}$. 

Our next aim is to provide a framework that upgrades the Feigin-Semikhatov duality at the critical level to an equivalence of certain blocks of representation categories. We formalize this situation as follows.

Let $L$ be a non-degenerate integral lattice of rank $r$, and $L^-$ be the lattice $ \sqrt{-1} L$.  
For any $L$-graded vertex algebra $V^1$, and lattice vertex algebra $V_L$, the torus $T:=\text{Hom}(L,\mathbb C^*)$ acts diagonally on $V^1\otimes V_L$.
The pair of vertex algebras $(W_{\infty}(\mathfrak g_2), W_{-h_1^\vee}(\mathfrak g_1,f_{\mathrm{sub}}))$ in Theorem \ref{thm:KScritical}  appears to be a special example for a more general framework between  $V^1$ and the orbifold vertex algebra
\begin{equation}
V^2\cong(V^1\otimes V_L)^T
\end{equation}
and  the reverse relation
\[
H^{\text{rel},0}(\mathfrak{gl}_1^{\oplus r}, V^2\otimes V_{L^-}) \cong V^1,
\]
 is just a consequence of Theorem \ref{thm:isoalgebra}.
 
Let $\mathfrak h$ be the Heisenberg algebra associated to $L$ and $L^*$ the dual lattice of $L$.
We also studied the abelian categories between the two vertex algebras $V^1$ and $V^2$. 
Let $\mathcal C_1$ be the category of $L$-graded
$V^1$-modules with grading-preserving morphisms.
For fixed $\mu\in L^*$, let $\mathcal C_2^\mu$
be the category of $V^2$-modules admitting the
Heisenberg decomposition
\[
N=\bigoplus_{\lambda\in L}
N_\lambda\otimes\pi_{\mu-\lambda},
\]
with $V^2$-module homomorphisms.
The isomorphisms mentioned above can be lifted to mutually inverse functors between abelian categories, namely for $\mu\in L^*$, one can define the orbifold functor $F_\mu$ and relative cohomology functor $G_\mu$ (see Corollary \ref{cor:func} for the precise definition) between the two categories $\cC_1$ and $\cC_2^\mu$,
\[
F_{\mu}: \cC_1\to \cC_2^\mu, \qquad G_\mu: \cC_2^\mu \to \cC_1.
 \]
We then show that the functors $F_\mu$ and $G_\mu$ are exact functors, which induces the equivalence for the categories. 
The shifting of the Heisenberg weight $\mu$ by the functor $F_{\mu}$ coincides with the spectral flow induced from the Li's $\Delta$-operator \cite{L},  on the corresponding module by $F_0$. Thus we obtain the following main theorem.
\begin{theorem} \label{thm:catequiv}
For $\mu\in L^*$, let $\sigma_{\mu}$ be the spectral flow induced by the $\Delta$-operator associated with $\mu$.
\begin{enumerate} 
\item The categories $\cC_1$ and $\cC_2^\mu$ are equivalent as abelian categories.
\item For any $M\in \cC_1$, there is a $V^2$-module isomorphism
\[
\sigma_\mu^*(F_0(M))\cong F_{\mu}(M).
\]
\end{enumerate}
\end{theorem}

Theorems~\ref{thm:KScritical} and~\ref{thm:catequiv} identify the principal block of the centerless critical-level $W$-algebra
$W_{-n-1}(\mathfrak{sl}_{n+1},f_{\mathrm{sub}})$ (resp. $W_{-2n+1}(\mathfrak{so}_{2n+1},f_{\mathrm{sub}})$)
with the principal block of the centerless large-level limit
$F_{n+1}:=W_{\infty}(\mathfrak{sl}_{n+1|1})$ (resp. $F_{2n}:=W_{\infty}(\mathfrak{osp}_{2|2n})$),
which is a free field algebra generated by two odd fields
$X^\pm(z)$
of conformal weight $\frac{n+2}{2}$ (resp. $\frac{2n+1}{2}$).

Let $\mathfrak g_n$ be the mode algebra of $F_n$, i.e. the $\mathbb Z$-graded Lie superalgebra spanned by the odd modes $X^\pm_m$ together with an even center. Let $Z(\mathfrak g_n)$ denote the supercenter of the universal enveloping algebra $U(\mathfrak g_n)$, which is an exterior algebra generated by $2n$ elements.
As explained in Section~\ref{sec:repncorr}, one can construct mutually inverse functors between the abelian category of finitely generated $\mathbb Z$-graded $F_n$-modules and the category of finite-dimensional $\mathbb Z$-graded $Z(\mathfrak g_n)$-modules, either by taking vacuum spaces or by induction.
Corollary~\ref{thm:classification} then shows that these two abelian categories are equivalent, and Remark~\ref{rm:loewy} implies that the Loewy structure of any finitely generated $F_n$-module is completely determined by the Loewy structure of its vacuum space as a $Z(\mathfrak g_n)$-module. Combining this with Theorems~\ref{thm:KScritical} and~\ref{thm:catequiv}, we obtain a complete description of the principal block of the centerless critical-level $W$-algebra $W_{-h_1^\vee}(\mathfrak{g}_{1},f_{\mathrm{sub}})$. For example, the algebra of self-extensions of $W_{-n}(\mathfrak{sl}_{n},f_{\mathrm{sub}}) $ (resp. $W_{-2n+1}(\mathfrak{so}_{2n+1},f_{\mathrm{sub}}) $) is isomorphic to the coordinate ring of the variety of $n \times n$ (respectively $2n \times 2n$) matrices whose $2 \times 2$ minors vanish, see Remark \ref{rem:ext}.
Intersecting this determinantal variety with
$\mathfrak{sl}_n\subset\operatorname{Mat}_n(\mathbb C)$
(respectively,
$\mathfrak{sp}_{2n}\subset\operatorname{Mat}_{2n}(\mathbb C)$)
gives the closure of the minimal nilpotent orbit of
$\mathfrak{sl}_n$ (respectively $\mathfrak{sp}_{2n}$).
\footnote{We thank S.~Nakatsuka for pointing this out.}
In type $A$, the subregular and minimal nilpotent orbits
are exchanged by Lusztig--Spaltenstein duality.

Since in particular the centerless subregular $W$-algebra $W_{-2}(\mathfrak{sl}_{2},f_{\mathrm{sub}})$ is the centerless affine vertex algebra $V_{-2}(\mathfrak{sl}_2)$,  we have the Loewy structure of its projective cover $P(V_{-2}(\mathfrak{sl}_2))$.

\begin{theorem}\label{thm:Loewy}
Let $V=V_{-2}(\mathfrak{sl}_2)$ be the centerless
critical-level affine vertex algebra, regarded as its vacuum
module, and let
$
V^{(m)}=\sigma^{m,*}(V),
$
where $\sigma$ is the spectral-flow transformation fixed in
Section~\ref{sec:spectral} and $\sigma^{m,*}$ denotes inverse
pullback by $\sigma^m$.
Then the projective cover
$
P\bigl(V_{-2}(\mathfrak{sl}_2)\bigr)
$
in $\Bcat_2$, defined in
Definition~\ref{def:BN-categories}, has the following
Loewy diagram.

\begin{center}
\resizebox{0.98\linewidth}{!}{%
\begin{tikzpicture}[->, thick, xscale=0.95, yscale=0.72]
    \node (1) at (0,6) {$V^{(0)}$};
    \node (21) at (-4.5,3) {$V^{(1)}$};
    \node (22) at (-1.5,3) {$V^{(1)}$};
    \node (23) at (1.5,3) {$V^{(-1)}$};
    \node (24) at (4.5,3) {$V^{(-1)}$};
    \node (31) at (-7.5,0) {$V^{(2)}$};
    \node (32) at (-4.5,0) {$V^{(0)}$};
    \node (33) at (-1.5,0) {$V^{(0)}$};
    \node (34) at (1.5,0) {$V^{(0)}$};
    \node (35) at (4.5,0) {$V^{(0)}$};
    \node (36) at (7.5,0) {$V^{(-2)}$};
    \node (41) at (-4.5,-3) {$V^{(1)}$};
    \node (42) at (-1.5,-3) {$V^{(1)}$};
    \node (43) at (1.5,-3) {$V^{(-1)}$};
    \node (44) at (4.5,-3) {$V^{(-1)}$};
    \node (5) at (0,-6) {$V^{(0)}$};

    \draw[red] (1) -- (21);
    \draw[green] (1) -- (22);
    \draw[blue] (1) -- (23);
    \draw[mauve] (1) -- (24);

    \draw[green] (21) -- (31);
    \draw[blue] (21) -- (32);
    \draw[mauve] (21) -- (33);
    \draw[red] (22) -- (31);
    \draw[blue] (22) -- (34);
    \draw[mauve] (22) -- (35);
    \draw[red] (23) -- (32);
    \draw[green] (23) -- (34);
    \draw[mauve] (23) -- (36);
    \draw[red] (24) -- (33);
    \draw[green] (24) -- (35);
    \draw[blue] (24) -- (36);

    \draw[blue] (31) -- (41);
    \draw[mauve] (31) -- (42);
    \draw[green] (32) -- (41);
    \draw[mauve] (32) -- (43);
    \draw[green] (33) -- (42);
    \draw[blue] (33) -- (43);
    \draw[red] (34) -- (41);
    \draw[mauve] (34) -- (44);
    \draw[red] (35) -- (42);
    \draw[blue] (35) -- (44);
    \draw[red] (36) -- (43);
    \draw[green] (36) -- (44);

    \draw[mauve] (41) -- (5);
    \draw[blue] (42) -- (5);
    \draw[green] (43) -- (5);
    \draw[red] (44) -- (5);

    \node[align=center] at (0,-7.25)


    {Loewy diagram of
    $P(V_{-2}(\mathfrak{sl}_2))$ in $\Bcat_2$};
\end{tikzpicture}%
}
\end{center}

Moreover,
$
V^{(\pm 1)}\cong D^\pm_{\mp\alpha},
$
where $D^\pm_\lambda$ denote the simple
module whose top level is the
highest (+) resp. lowest (-) weight module of  $\mathfrak{sl}_2$ of highest/lowest weight $\lambda$, and $\alpha$ is the simple root.
Let $H$ be the standard Hamiltonian on $V$, assigning
weight one to the affine currents, and put $J=h/2$.
Equip $V^{(m)}$ with the compatible Hamiltonian
$
H^{(m)}=H+mJ_0,
$
where the operators on the right act on the original
module $V$. Then
\begin{itemize}
    \item
    $(V^{(0)},H^{(0)})$ belongs to
    $V$-\textup{wtmod}$_{KL}$;

    \item
    $(V^{(1)},H^{(1)})$ belongs to
    $V$-\textup{wtmod}$_{\mathcal O}$ but not to
    $V$-\textup{wtmod}$_{KL}$;

    \item
    $(V^{(-1)},H^{(-1)})$ belongs to
    $V$-\textup{wtmod}$_{\geq}$ but not to
    $V$-\textup{wtmod}$_{\mathcal O}$;

    \item
    $(V^{(\pm2)},H^{(\pm2)})$ belongs to
    $V$-\textup{wtmod} but is not lower bounded.
\end{itemize}
\end{theorem}

Previously the structure of representations of centerless affine vertex algebras had been studied in the much smaller category $V_{-h^\vee}(\mathfrak{g})$-\textup{wtmod}$_{\mathcal O}$, see the works of Arakawa and Fiebig \cite{AF1, AF2, Fi}. Beyond this category there are studies by Adamovic et al with a focus on realizing simple modules \cite{A2, A4, AJN}.

We conclude with a list of problems and questions that are connected to our studies. 

\subsubsection{Open questions and future directions}

\begin{enumerate}
    \item At present we can describe the principal block of the centerless subregular $W$-algebras and our intention is to obtain similar results for other blocks and to include representations with non-trivial action of the center. The idea is to replace the free field superalgebras by larger vertex superalgebras that contain central extensions of the free field superalgebras as quotients. 
    \item The categorical framework to study the critical level Feigin-Semikhatov dualities deserves further attention. This framework should apply to more interesting examples and it is important to study the involved functors on intertwining operators. 
    \item The overall goal is to identify $W_{-h^\vee}(\mathfrak g, f_{\text{sub}})$-wtmod or a larger subcategory of  $W^{-h^\vee}(\mathfrak g, f_{\text{sub}})$-wtmod with a category of modules of a certain quantum supergroup. $V_k(\mathfrak{sl}_2)$-wtmod for $k$ admissible is braided tensor equivalent to a so-called partial semisimplification of a twisted Deligne product of categories of quantum $\mathfrak{sl}_2$ and $\mathfrak{sl}_{2|1}$ at roots of unities related to $k$ \cite{CL-KL1}. The techniques of \cite{CL-KL1, CL-KL2} hopefully also apply to these cases and one naturally expects quantum supergroups that are somehow related to degenerate versions of $\mathfrak{sl}_{n|1}$ or $\mathfrak{osp}_{2|2n}$. A hint comes 
    from the mirabolic/orthosymplectic Satake equivalence: equivalences between certain categories of  equivariant perverse sheaves on affine Grassmannians and categories of degenerate versions of supergroups \cite{AFGT, AFT}.  
\item Only recently the first results on the center of affine vertex superalgebras at the critical level have been obtained \cite{A3, AN, AFN} and using our methods it should be possible to say more about principal $W$-superalgebras of $\mathfrak{sl}_{n|1}$ at the critical level. 
\item For $n=1$, the affine Feigin-Tipunin algebra $FT[\mathfrak{gl}_{m},1]$ is also expected to be related to global sections of the sheaf $\mathcal D^{\text{ch}}_{X, -m+1}$ of chiral differential operators on the base affine space $X:=SL_{m}/N$ at the level $-m+1$, where $N$ is the maximal unipotent subgroup of $SL_{m}$. 
More precisely, we expect the following isomorphism
\begin{equation}\label{eq:chiralFT}
FT[\mathfrak{gl}_{m},1]\cong H^{\mathrm{rel},0}(\mathfrak{gl}_1^{\oplus m-1}, \Gamma(X,\mathcal D^{\text{ch}}_{X, -m+1}) \otimes \mathcal E(m)),
\end{equation} 
where  $\mathcal E(m)$ is the rank $m$ $bc$-system. 
We note that $V_m:=\Gamma(X,\mathcal D^{\text{ch}}_{X, -m+1})$ is isomorphic to the rank two $\beta\gamma$-system when $m=2$ \cite{ADS1, ADS2}, and hence the right hand side of  (\ref{eq:chiralFT}) coincides with $FT[\mathfrak{gl}_2,1] \cong L_1(\mathfrak{psl}_{2|2})$.
Furthermore the associated variety $X_{V_m}$ of $V_m$ is shown to be isomorphic to the affine closure of the cotangent bundle $T^*(X)$ \cite{ADS2}, which is a symplectic singularity.  In particular $V_m$ is  quasi-lisse. It is natural to conjecture that the associated variety of $FT[\mathfrak{gl}_{m},1]$ is isomorphic to the Hamiltonian reduction (equivalently, $(\mathbb C^*)^{m-1}$-reduction) of $X_{V_m}$, and  $FT[\mathfrak{gl}_{m},1]$ is also quasi-lisse.
\end{enumerate}

\section{Background}
\subsection{Vertex-superalgebra conventions}

We work over $\mathbb C$ and use the conventions of
\cite{K}. A vertex superalgebra is a $\mathbb Z_2$-graded
vector space $V=V_{\bar0}\oplus V_{\bar1}$ equipped with
an even vacuum vector $\vac$, an even translation
operator $T$, and a parity-preserving state-field map
\[
Y(\,\cdot\,,z):
V\longrightarrow(\operatorname{End}V)[[z,z^{-1}]],
\qquad
Y(a,z)=\sum_{m\in\mathbb Z}a_{(m)}z^{-m-1}.
\]
These data satisfy the truncation condition
$a_{(m)}b=0$ for $m \mathrel{>\!>} 0$, the vacuum and translation
axioms
\[
Y(\vac,z)=\operatorname{id}_V,\qquad
Y(a,z)\vac\in a+zV[[z]],\qquad
T\vac=0,\qquad
[T,Y(a,z)]=\partial_zY(a,z),
\]
and locality: for homogeneous $a,b\in V$, there exists
$N\geq0$ such that
\[
(z-w)^N[Y(a,z),Y(b,w)]=0.
\]
Here and below, brackets denote supercommutators, and
$p(a)\in\mathbb Z_2$ denotes the parity of a homogeneous
element $a$. A vertex algebra is the purely even case.

All modules are untwisted weak modules unless further
conditions are stated; in particular, their vertex
operators satisfy the field truncation condition.
Morphisms of supermodules preserve parity. We will use
the iterate identity
\begin{equation}\label{2.4}
\begin{aligned}
(a_{(n)}b)_{(m)}
={}&
\sum_{i\geq0}(-1)^i\binom ni
a_{(n-i)}b_{(m+i)}
\\
&-
(-1)^{p(a)p(b)+n}
\sum_{i\geq0}(-1)^i\binom ni
b_{(m+n-i)}a_{(i)},
\end{aligned}
\end{equation}
valid on $V$ and its modules, with the sums interpreted
after application to a vector.

A Hamiltonian is an even diagonalizable operator $H$
on $V$ satisfying
\begin{equation}\label{2.2}
[H,Y(a,z)]
=z\partial_zY(a,z)+Y(Ha,z).
\end{equation}
If $Ha=\Delta a$, then
\[
[H,a_{(m)}]=(\Delta-m-1)a_{(m)}.
\]
A \emph{compatible module Hamiltonian} is an even
operator $H_M$ satisfying
\[
[H_M,Y_M(a,z)]
=z\partial_zY_M(a,z)+Y_M(Ha,z).
\]
When used below, $H_M$ is required to be locally finite,
so that $M$ is the direct sum of its generalized
eigenspaces. In general, a compatible module
Hamiltonian is additional structure.
When $H=L_0$ comes from a specified conformal vector,
the module action supplies the corresponding operator
$L_0^M$, although its local finiteness is an additional
condition. Unless stated otherwise, our category
equivalences and extension groups are taken without
a chosen module Hamiltonian.

A strong generating set spans $V$ by normally ordered
monomials in its derivatives, where
\[
:ab:\,=a_{(-1)}b.
\]
A free strong generating set gives a PBW basis after
ordering the pairs consisting of a generator and a
derivative order. Identical odd derivative generators
cannot be repeated, whereas different derivatives of
the same odd generator may occur together.
Polynomial algebras on supervariables are understood
to be supercommutative.

\subsection{Free fields and lattices}

\begin{defn}\label{defn:freefield}
A free-field vertex superalgebra is generated by fields
with scalar singular OPE coefficients.
For generators $X^i$ homogeneous under a Hamiltonian,
we write
\[
X^i(z)X^j(w)
\sim
a_{ij}(z-w)^{-\operatorname{wt}X^i-\operatorname{wt}X^j},
\qquad a_{ij}\in\mathbb C,
\]
where $a_{ij}=0$ unless
$\operatorname{wt}X^i+\operatorname{wt}X^j$
is a positive integer, and the coefficients satisfy
the symmetry imposed by parity and skew-symmetry.
We use the usual vacuum free-field algebras with
their PBW bases.
\end{defn}

The generalized orthogonal and symplectic free-field
algebras used below are those of \cite{CL1}.
In particular, $F_N$ has two odd generators $X^\pm$
of charges $\pm1$ and weight $(N+1)/2$, with
\[
X^+(z)X^-(w)\sim(z-w)^{-N-1}.
\]
For odd $N$ the scalar pairing is symplectic, and
for even $N$ it is orthogonal. The nondegenerate
oscillator pairing gives simplicity of $F_N$.

A rank-one Heisenberg field $a$ of level
$\kappa\neq0$ satisfies
\[
a(z)a(w)\sim\kappa(z-w)^{-2}.
\]
We denote its Fock module of zero-mode eigenvalue
$\lambda$ by $\pi_\lambda$, with chosen highest-weight
vector $e^\lambda$ satisfying
\[
a_{(0)}e^\lambda=\lambda e^\lambda,
\qquad
a_{(n)}e^\lambda=0\quad(n>0).
\]
An even current $a$ satisfying $a(z)a(w)\sim0$
is called an \emph{isotropic current}.

The rank-one charged free-fermion system is generated
by odd fields $b,c$ with
\[
b(z)c(w)\sim(z-w)^{-1}.
\]
For BRST ghosts we use the mode expansions
\[
b(z)=\sum_{m\in\mathbb Z}b_mz^{-m-1},
\qquad
c(z)=\sum_{m\in\mathbb Z}c_mz^{-m},
\]
and Hamiltonian weights $1,0$, respectively.
The same underlying vertex superalgebra also admits
weights $1/2,1/2$; we use this grading for the copy
identified with $V_{\mathbb Z}$.

Let $L$ be a nondegenerate integral lattice and put
$\mathfrak h=L\otimes_{\mathbb Z}\mathbb C$.
We write
\[
V_L=\bigoplus_{\lambda\in L}\pi_\lambda
\]
for its lattice vertex superalgebra.
Here the vector label $\lambda$ specifies the Fock
momentum:
\[
h_{(0)}e^\lambda=(h,\lambda)e^\lambda,
\qquad
h_{(n)}e^\lambda=0
\quad(h\in\mathfrak h,\ n>0).
\]
The parity of $e^\lambda$ is
$(\lambda,\lambda)\bmod2$, and positive definiteness
of $L$ is not assumed.
Cocycles for $L$ and the opposite lattice $L^-=iL$
are chosen so that, after a coboundary adjustment,
the restricted cocycle on the even isotropic diagonal
subgroup of $L\oplus L^-$ is trivial.
This normalization is used in the BRST reconstruction
maps.

\subsection{Affine vertex algebras}

Let $\mathfrak{g}$ be a finite-dimensional simple Lie superalgebra equipped with an invariant bilinear form $(\cdot,\cdot)$.
The associated affine Lie algebra is
\[
\widehat{\mathfrak{g}} = \mathfrak{g}[t,t^{-1}] \oplus \mathbb{C}K,
\]
with commutation relations for any $\xi,\eta\in \gg$
\[
[\xi \otimes t^m, \eta \otimes t^n]
= [\xi,\eta] \otimes t^{m+n} + m\,\delta_{m,-n}(\xi,\eta)K.
\]

Fix a level $k \in \mathbb{C}$. The universal affine vertex algebra $V^k(\mathfrak{g})$ is defined as the vacuum module of $\widehat{\mathfrak{g}}$ at level $k$,
\[
V^k(\mathfrak{g}) = \mathrm{Ind}_{\mathfrak{g}[t] \oplus \mathbb{C}K}^{\widehat{\mathfrak{g}}} \mathbb{C},
\]
where $\mathfrak{g}[t]$ acts trivially on $\mathbb C$ and $K$ acts by scalar $k$.
It is strongly generated by fields $J^{\xi}(z)$ for $\xi \in \mathfrak{g}$ with OPEs
\[
J^{\xi}(z)J^{\eta}(w) \sim \frac{k(\xi,\eta)}{(z-w)^2} + \frac{J^{[\xi,\eta]}(w)}{z-w}.
\]

\subsection{Affine \texorpdfstring{$W$}{W}-algebras and their Hamiltonians}
Let $f\in\mathfrak g_{\bar0}$ be nilpotent and choose a good grading
$\mathfrak g=\bigoplus_{j\in\frac12\mathbb Z}\mathfrak g_j$ with $f\in\mathfrak g_{-1}$. Write $x$ for its grading element, so $[x,u]=ju$ for $u\in\mathfrak g_j$, and put $\mathfrak g_+=\bigoplus_{j>0}\mathfrak g_j$. For a homogeneous basis $u_\alpha\in\mathfrak g_{j_\alpha}$, the charged ghosts $b^\alpha,c_\alpha$ have parity $p(u_\alpha)+1$. Thus in the superalgebra case the charged ghost system need not consist entirely of odd fields. The neutral system on $\mathfrak g_{1/2}$ uses the form $(f,[u,v])$ and fields $\Phi_u$ of parity $p(u)$.

The quantum Drinfeld--Sokolov complex is
\[
C_{\mathrm{DS}}=V^k(\mathfrak g)\otimes\mathcal E_{\mathrm{ch}}(\mathfrak g_+)
\otimes\mathcal E_{\mathrm{ne}}(\mathfrak g_{1/2}).
\]
We use the charged-ghost pairings and signs of \cite{KRW,KW}. With these conventions its differential is the zero mode of the odd BRST field
\begin{align*}
d={}&\sum_\alpha(-1)^{p(u_\alpha)}:J^{u_\alpha}c_\alpha:
-\frac12\sum_{\alpha,\beta}(-1)^{p(u_\alpha)(1+p(u_\beta))}
:b^{[u_\alpha,u_\beta]}c_\alpha c_\beta:\\
&+\sum_\alpha(f,u_\alpha)c_\alpha
+\sum_{j_\alpha=1/2}:c_\alpha\Phi_{u_\alpha}:.
\end{align*}
Here $b^u$ depends linearly on $u$. The differential satisfies $d_0^2=0$, and
\[
W^k(\mathfrak g,f)=H^0(C_{\mathrm{DS}},d_0).
\]
The PBW theorem of \cite{KW} gives free strong generators indexed by a homogeneous basis of $\mathfrak g^f$, with the corresponding supercommutative associated graded algebra.

We define the grading uniformly in $k$.
Let $D_{\mathrm{aff}}$ be the standard grading operator
on $V^k(\mathfrak g)$, characterized by
\[
D_{\mathrm{aff}}\vac=0,\qquad
[D_{\mathrm{aff}},J^u_{(n)}]=-nJ^u_{(n)}.
\]
On the affine factor, use $D_{\mathrm{aff}}-x_{(0)}$;
thus $J^u$ has weight $1-j$ for $u\in\mathfrak g_j$.
For each charged ghost pair associated with
$u\in\mathfrak g_j$, $j>0$, assign $b^u$ and $c_{u^*}$
weights $1-j$ and $j$, respectively.
Each neutral ghost $\Phi_u$, with
$u\in\mathfrak g_{1/2}$, has weight $1/2$.
These assignments respect the OPEs and define a Hamiltonian $H_{\mathrm{DS}}$ on the complex. Each term of $d$ has weight one: the first and cubic terms have total weight one, $(f,u_\alpha)\neq0$ implies $j_\alpha=1$, and the neutral term has weight $1/2+1/2$. Therefore
\[
[H_{\mathrm{DS}},d_0]=0.
\]
The induced operator $H$ on $W^k(\mathfrak g,f)$ is consequently defined for all $k$, including $k=-h^\vee$. A generator corresponding to $u\in\mathfrak g^f\cap\mathfrak g_j$ has weight $1-j$. At noncritical levels this is the grading of the usual DS conformal vector. At the critical level the grading remains meaningful even though that conformal vector is unavailable.
The augmentation ideal of the critical center is homogeneous for $H$, so $H$ descends to the centerless quotient. 
All weights of critical-level fields below refer to this Hamiltonian.

\section{Relative semi-infinite cohomology}\label{sec:BRST}
Let $K(z)=\sum_{m\in\mathbb Z}K_mz^{-m-1}$ be an even
isotropic current in a vertex superalgebra $V$, and let
$\mathcal E$ be the BRST ghost system of Section~2.
For a $V$-module $M$ put
\[
C(M)=M\otimes\mathcal E,\qquad
Q=\operatorname{Res}_z:c(z)K(z):
=\sum_{m\in\mathbb Z}c_{-m}K_m.
\]
The OPE $K(z)K(w)\sim0$ gives $Q^2=0$.
Give $c$ cohomological degree $1$ and $b$ degree $-1$.
Since $[Q,b_0]=K_0$, the relative subcomplex
\[
C^{\mathrm{rel}}(M)
=\ker K_0\cap\ker b_0\subset C(M)
\]
is stable under $Q$. Its cohomology is denoted by
$H^{\mathrm{rel},p}(M)$. The zero-mode differential is
a vertex-superalgebra derivation on $C(V)$; consequently
$H^{\mathrm{rel},0}(V)$ is a vertex superalgebra and the
cohomology of $M$ is naturally a module for it.
For a rank-$r$ isotropic system we use $r$ ghost pairs
and the sum of the commuting rank-one differentials.

\begin{lemma}[Fock-module cohomology]\label{lem:vanish}
Let $J^+,J^-$ be commuting Heisenberg fields of levels
$\kappa,-\kappa$, where $\kappa\neq0$, and set
$K=J^++J^-$.
Label their Fock modules by their zero-mode eigenvalues.
Write $e^\lambda$ and $e^\mu$ for chosen
highest-weight vectors in $\pi^+_\lambda$ and
$\pi^-_\mu$, respectively.
Then
\[
H^{\mathrm{rel},p}
(\pi^+_\lambda\otimes\pi^-_\mu)
=
\delta_{p,0}\delta_{\lambda+\mu,0}\,
\mathbb C[e^\lambda\otimes e^\mu].
\]
The corresponding statement holds for opposite
nondegenerate Heisenberg forms of any finite rank.
\end{lemma}

\begin{proof}
If $\lambda+\mu\neq0$, the relative subcomplex is zero.
Otherwise put
$A=(J^+-J^-)/(2\kappa)$, so
$K(z)A(w)\sim(z-w)^{-2}$.
The relative complex is the tensor product of the
highest-weight line with
\[
\mathbb C[K_{-n},A_{-n}\mid n\geq1]
\otimes\Lambda(b_{-n},c_{-n}\mid n\geq1).
\]
The relative ghost condition excludes $c_0$.
On the displayed generators,
\[
Qb_{-n}=K_{-n},\qquad
QA_{-n}=n c_{-n},\qquad
QK_{-n}=Qc_{-n}=0.
\]
Define an odd operator $s$ by
\[
sK_{-n}=b_{-n},\qquad
sc_{-n}=A_{-n}/n,\qquad
sA_{-n}=sb_{-n}=0,
\]
extending as a superderivation.
Then $Qs+sQ$ is the total polynomial/exterior-degree
operator. Every vector is a finite sum of monomials,
so division by that degree contracts its positive-degree
part. Only the highest-weight line survives, in
cohomological degree zero.
The higher-rank case is the tensor product of these
contractions after choosing dual bases.
This is the standard Fock calculation underlying
\cite{FGZ,CGNS}.
\end{proof}

\begin{prop}\label{prop:relcoh}
Suppose $V$ contains the commuting Heisenberg tensor
product $\pi^+\otimes\pi^-$ and has a decomposition
\[
V=\bigoplus_{\lambda,\mu}
\pi^+_\lambda\otimes\pi^-_\mu\otimes M_{\lambda,\mu}
\]
into ordinary Fock modules with multiplicity spaces
$M_{\lambda,\mu}$.
Write
\[
V^{K[t]}
=\{v\in V:K_nv=0\text{ for all }n\geq0\}.
\]
The map
\begin{equation}\label{map:surj}
V^{K[t]}\longrightarrow H^{\mathrm{rel},0}(V),
\qquad
v\longmapsto[v\otimes\vac],
\end{equation}
is a surjective vertex-superalgebra homomorphism whose
kernel is the ideal generated by $K$.
Moreover $H^{\mathrm{rel},p}(V)=0$ for $p\neq0$.
\end{prop}

\begin{proof}
The differential acts trivially on every multiplicity
space. Lemma~\ref{lem:vanish} therefore gives
\[
H^{\mathrm{rel},p}(V)
=
\delta_{p,0}\bigoplus_\lambda
\mathbb C[e^\lambda\otimes e^{-\lambda}]
\otimes M_{\lambda,-\lambda}.
\]
These highest-weight representatives are annihilated
by all $K_n$, $n\geq0$, proving surjectivity.
To identify the kernel, use the variables
$K_{-n},A_{-n}$ in the proof of the lemma.
On a matching Fock summand, $K_n$ for $n>0$ acts as
$n\partial/\partial A_{-n}$. Hence
\[
V^{K[t]}
=
\bigoplus_\lambda
\mathbb C[K_{-1},K_{-2},\ldots]
(e^\lambda\otimes e^{-\lambda})
\otimes M_{\lambda,-\lambda}.
\]
The map~\eqref{map:surj} is evaluation at $K_{-n}=0$
on each summand. Its kernel is the span of the
positive-degree $K$-monomials. Since
\[
K_{-n}v
=\frac{(T^{n-1}K)_{(-1)}v}{(n-1)!},
\]
this span is precisely the vertex ideal generated
by $K$ in the commutant.
The multiplication and translation on the quotient
agree with cohomology because $Q$ is a vertex-algebra
derivation. This proves all assertions.
\end{proof}

\section{Lattice reconstruction and equivalence of categories}

\subsection{Equivalence of $L$-graded categories}\label{sec:equiv}
Let $L$ be a nondegenerate integral lattice, $\mathfrak h=L\otimes_{\mathbb Z}\mathbb C$, and $T=\operatorname{Hom}(L,\mathbb C^*)$. Let
$V^1=\bigoplus_{\lambda\in L}V^1_\lambda$ be an $L$-graded vertex superalgebra. The torus acts on $V^1_\lambda$ and $\pi_\lambda\subset V_L$ by the character $\lambda$, and
\[
V^2=(V^1\otimes V_L)^T
=\bigoplus_{\lambda\in L}V^1_\lambda\otimes\pi_{-\lambda}.
\]
In tensoring with $V_{L^-}$, the BRST current is the sum of the two oppositely paired Heisenberg currents. Opposite momenta are identified so that
$e^{-\lambda}\otimes e^{i\lambda}$ has zero total charge. Use the compatible cocycles fixed in Section~2. Ghost factors are implicit in $H^{\mathrm{rel},p}$.

\begin{theorem}\label{thm:isoalgebra}
There are isomorphisms
\[
H^{\mathrm{rel},p}(V^2\otimes V_{L^-})\cong\delta_{p,0}V^1,
\]
where the isomorphism in degree zero is an isomorphism of vertex superalgebras.
\end{theorem}
\begin{proof}
The tensor product is the direct sum of
$V^1_\lambda\otimes\pi_{-\lambda}\otimes\pi_{i\nu}$, $\lambda,\nu\in L$.
The relative charge condition selects $\nu=\lambda$, and Lemma~\ref{lem:vanish} contracts the nonzero modes and the ghosts. Thus the linear map
\[
\Phi(v_\lambda)=[v_\lambda\otimes e^{-\lambda}\otimes e^{i\lambda}]
\]
is bijective, and higher cohomology vanishes.

We check the product, rather than only its underlying vector space. For $v_\lambda\in V^1_\lambda$ and $w_\nu\in V^1_\nu$, the two lattice vertex operators contribute powers
$z^{(\lambda,\nu)}$ and $z^{-(\lambda,\nu)}$, which cancel. The tensor-product super sign and the lattice cocycles cancel with our isotropic-diagonal normalization. The remaining creation factor is an exponential of negative modes of the diagonal isotropic currents. Each such mode is $Q$-exact by $Qb_{-n}=K_{-n}$. Therefore, coefficientwise in the formal variable,
\[
Y(\Phi(v_\lambda),z)\Phi(w_\nu)
=[Y(v_\lambda,z)w_\nu\otimes e^{-\lambda-\nu}\otimes e^{i(\lambda+\nu)}]
=\Phi(Y(v_\lambda,z)w_\nu).
\]
The lattice contributions to translation are likewise diagonal isotropic creation modes and hence exact. The vacuum is preserved. This proves the vertex-superalgebra assertion.
\end{proof}

\begin{theorem}\label{thm:isoalgebra2}
Suppose $V$ contains the Heisenberg algebra of $L$ and decomposes as
$V=\bigoplus_{\lambda\in L}C_\lambda\otimes\pi_\lambda$.
Then
\[
\bigl(H^{\mathrm{rel},p}(V\otimes V_{L^-})\otimes V_L\bigr)^T
\cong\delta_{p,0}V.
\]
The torus acts on the cohomology class from $C_\lambda\otimes\pi_\lambda\otimes\pi_{-i\lambda}$ with character $-\lambda$.
\end{theorem}
\begin{proof}
For $v\in C_\lambda$ and $s\in\pi_\lambda$, the vector-space isomorphism is
\[
v\otimes s\longmapsto
[v\otimes e^{\lambda}\otimes e^{-i\lambda}]\otimes s.
\]
For a Heisenberg-highest vector $a$ of momentum
$\alpha$, the identity
\[
[h_{(n)},Y(a,z)]=(h,\alpha)z^nY(a,z)
\]
determines its Heisenberg creation and annihilation
factors. Indeed, removing these factors gives
\[
\exp\!\left(-\sum_{n>0}\frac{\alpha_{(-n)}}{n}z^n\right)
Y(a,z)
\exp\!\left(\sum_{n>0}\frac{\alpha_{(n)}}{n}z^{-n}\right),
\]
which commutes with all nonzero Heisenberg modes
and hence acts on the Fock multiplicity spaces.
Tensoring with the opposite lattice cancels the
Heisenberg factors in relative cohomology, as in
Theorem~\ref{thm:isoalgebra}; tensoring back with
$V_L$ restores them, including the zero-mode powers
and the fixed cocycles. Thus the displayed map
intertwines the fields of all Heisenberg-highest
vectors. Their Heisenberg descendants span $V$,
so it intertwines every field.
The vacuum and translation are preserved as in
Theorem~\ref{thm:isoalgebra}, and higher cohomology
vanishes summandwise.
\end{proof}

We now lift the reconstruction isomorphisms to equivalences
of module categories.

\begin{defn}\label{def:chargecategories}
Let $\mathcal C_1$ be the category of $L$-graded
$V^1$-modules
\[
M=\bigoplus_{\lambda\in L}M_\lambda
\]
such that
\[
a_{(n)}M_\lambda\subseteq M_{\lambda+\alpha},
\qquad
a\in V^1_\alpha,\quad n\in\mathbb Z.
\]
Morphisms are grading-preserving $V^1$-module
homomorphisms.
\end{defn}

\begin{defn}\label{def:Fockcategory}
For a fixed $\mu\in L^*$, let $\mathcal C_2^\mu$
be the category of $V^2$-modules whose restriction to
the Heisenberg vertex algebra admits a decomposition
\[
N=\bigoplus_{\lambda\in L}
N_\lambda\otimes\pi_{\mu-\lambda},
\]
where $N_\lambda$ denotes the corresponding Fock
multiplicity space. Morphisms are $V^2$-module
homomorphisms.
\end{defn}

Since the Heisenberg Fock modules with distinct momenta
are nonisomorphic, every morphism in $\mathcal C_2^\mu$
preserves this decomposition and acts on each summand
as a map of multiplicity spaces tensored with the
identity. Moreover,
\[
\bigl(V^1_\alpha\otimes\pi_{-\alpha}\bigr)_{(n)}
\bigl(N_\lambda\otimes\pi_{\mu-\lambda}\bigr)
\subseteq
N_{\lambda+\alpha}\otimes\pi_{\mu-\lambda-\alpha}.
\]
Thus the fixed representative $\mu$ determines a
compatible $L$-grading on $N$.

Both categories are abelian, i.e. kernels and cokernels
are computed on the graded components in $\mathcal C_1$
and on the Fock multiplicity spaces in $\mathcal C_2^\mu$.
All exactness and projectivity statements below refer
to these categories. Their definitions require no
choice of a compatible module Hamiltonian.

For $\mu\in L^*$, the lattice module
\[
V_{\mu+L}=\bigoplus_{\beta\in L}\pi_{\mu+\beta}
\]
is an untwisted $V_L$-module. Give the summand
$\pi_{\mu+\beta}$ torus character $\beta$,
\[
t\cdot u=t(\beta)u,
\qquad
u\in\pi_{\mu+\beta}.
\]
We use compatible lattice cocycles and highest-weight
vector normalizations as in
Theorem~\ref{thm:isoalgebra}.

\begin{prop}\label{prop:equivariant-module}
Let $M$ be an $L$-graded $V^1$-module and let
$\mu\in L^*$. Then $M\otimes V_{\mu+L}$ is a
$T$-equivariant $V^1\otimes V_L$-module.
Consequently,
\[
(M\otimes V_{\mu+L})^T
\]
carries a natural $V^2$-module structure.
\end{prop}

\begin{proof}
The diagonal torus action is
\[
t\cdot(m_\lambda\otimes u_{\mu+\beta})
=
t(\lambda+\beta)\,
m_\lambda\otimes u_{\mu+\beta},
\]
where $m_\lambda\in M_\lambda$ and
$u_{\mu+\beta}\in\pi_{\mu+\beta}$.
The module actions respect the charge gradings.
Therefore
\[
t\cdot Y(a,z)w
=
Y(t\cdot a,z)(t\cdot w)
\]
for $a\in V^1\otimes V_L$ and
$w\in M\otimes V_{\mu+L}$.
If both $a$ and $w$ are $T$-invariant, so is every
coefficient of $Y(a,z)w$. Hence the invariant subspace
is stable under $V^2=(V^1\otimes V_L)^T$.
\end{proof}

\begin{cor}\label{cor:func}
For fixed $\mu\in L^*$, there is a functor
\[
F_\mu : \mathcal{C}_1 \to \mathcal{C}_2^\mu,
\quad
M \mapsto (M \otimes V_{\mu+L})^T,
\]
and one has a vector space isomorphism
\[
F_\mu(M)
\cong
\bigoplus_{\lambda \in L} M_\lambda \otimes \pi_{-\lambda+\mu}.
\]
Similarly there is a functor
\[
G_{\mu }: \mathcal{C}_2^\mu \longrightarrow \mathcal{C}_1, \qquad N \longmapsto H^{\mathrm{rel},0}(N \otimes V_{-i\mu+L^-}).
\]
\end{cor}

\begin{proof}
Proposition~\ref{prop:equivariant-module} gives the
$V^2$-module structure on $F_\mu(M)$.
A summand
$M_\lambda\otimes\pi_{\mu+\beta}$
is invariant precisely when $\lambda+\beta=0$.
This proves the displayed Fock decomposition and
shows that $F_\mu(M)\in\mathcal C_2^\mu$.

For $N\in\mathcal C_2^\mu$, the tensor product
$N\otimes V_{-i\mu+L^-}$ is a module for
$V^2\otimes V_{L^-}$. Its relative cohomology is
therefore a module for
\[
H^{\mathrm{rel},0}(V^2\otimes V_{L^-})
\cong V^1
\]
by Theorem~\ref{thm:isoalgebra}.
Lemma~\ref{lem:vanish} gives
\[
H^{\mathrm{rel},p}
\bigl(N\otimes V_{-i\mu+L^-}\bigr)
\cong
\begin{cases}
\displaystyle\bigoplus_{\lambda\in L}N_\lambda,
& p=0,\\
0,&p\ne0.
\end{cases}
\]
More explicitly, the summand $N_\lambda$ is represented by
\[
\left[
u_\lambda\otimes e^{\mu-\lambda}
\otimes  e^{ i \lambda-i\mu}
\right],
\qquad   u_\lambda\in N_\lambda.
\]
Assign this class degree $\lambda$.
The action of $V^1_\alpha$ shifts this degree by
$\alpha$, so $G_\mu(N)\in\mathcal C_1$.

The constructions commute with module homomorphisms.
Finally, $F_\mu$ is exact componentwise, and the
displayed description of $G_\mu$ shows that it is
exact on Fock multiplicity spaces.
\end{proof}

By Theorem  \ref{thm:isoalgebra}, when $\mu=0$, we have isomorphisms of graded vertex algebras $G_\mu\circ F_\mu (V^1) \cong V^1$ and $F_\mu \circ G_\mu (V^2) \cong V^2$.
Similarly, we have the following  isomorphisms for $L$-graded vertex algebra modules.
\begin{theorem} \label{thm:iso}
For any $\mu\in L^*$, $M \in \mathcal{C}_1$, there is an isomorphism
\begin{equation} \label{eq:C1}
M \;\cong\;G_\mu\circ F_\mu (M) =
H^{\mathrm{rel},0}\!\big((M \otimes  V_{\mu+L})^T \otimes  V_{-i \mu+L^-} \big)
\end{equation}
of $L$-graded $V^1$-modules. 
Similarly, for any $N \in \mathcal{C}_2^\mu$, there is an isomorphism
\begin{equation}\label{eq:C2}
N \;\cong\;F_\mu \circ G_\mu (N)=
\big( H^{\mathrm{rel},0}(N \otimes V_{- i \mu+L^-})  \otimes V_{\mu+L} \big)^T 
\end{equation}
of $L$-graded $V^2$-modules.
\end{theorem}

\begin{proof}
 We define a linear map
\begin{align*}
\Phi_M:\; 
M &\longrightarrow H^{\mathrm{rel},0}\!\big((M \otimes V_{\mu+L})^T \otimes V_{-i\mu+L^-} \big),\\
m_\lambda 
&\longmapsto 
\big[\, m_\lambda  
   \otimes   e^{-\lambda+\mu}
   \otimes  e^{i\lambda-i\mu} \,\big].
\end{align*}
which is an isomorphism of $L$-graded $V^1$-modules similar to the proof of Theorem \ref{thm:isoalgebra}.

 By definition, $N=\oplus_{\lambda\in L}  N_{\lambda}\otimes \pi_{-\lambda+\mu}$, where $N_{\lambda} $ is defined to be the  multiplicity of $\pi_{-\lambda+\mu}$ in $N$. 
Then the isomorphism in~\eqref{eq:C2} can be proved similarly by constructing the map
\begin{align*}
\Psi_N:\;
N &\longrightarrow 
\big( H^{\mathrm{rel},0}(N \otimes V_{-i\mu+L^-}) \otimes V_{\mu+L} \big)^T,\\
v_\lambda \otimes s_{-\lambda+\mu} 
&\longmapsto 
\big[\, v_\lambda \otimes e^{-\lambda+\mu}
   \otimes  e^{i \lambda-i\mu} \,\big]
   \otimes  s_{-\lambda+\mu}.
\end{align*}
This map is bijective on each Fock summand.
The field-factorization argument in the proof of
Theorem~\ref{thm:isoalgebra2}, applied to module
fields, shows that it intertwines the $V^2$-action.
Since $\mathcal C_1$ and $\mathcal C_2^\mu$ are abelian,
the equivalences $F_\mu$ and $G_\mu$ are exact.
\end{proof}

\begin{theorem} \label{thm:equivalence}
For $\mu\in L^*$, the functors $F_\mu$ and $G_\mu$ form mutually inverse functors, namely there exists natural exact equivalences of functors
\begin{equation}
G_\mu \circ F_\mu \;\cong\; \mathrm{Id}_{\mathcal{C}_1}, 
\qquad
F_\mu \circ G_\mu \;\cong\; \mathrm{Id}_{\mathcal{C}_2^\mu}.
\end{equation}
 In particular, we have an equivalence of $\cC_1$ and $\cC_2^\mu$ as abelian categories.
\end{theorem}

\begin{proof}
By Theorem~\ref{thm:iso}, it suffices to show that the functors constructed therein are natural. 
This follows by tracing the elements in the commutative diagrams below,
\[
\begin{tikzcd}
M \arrow[r, "\Phi_M"] \arrow[d, "f"'] 
  & G_\mu \circ F_\mu(M)
      \arrow[d, "G_\mu \circ F_\mu(f)"] \\
M' \arrow[r, "\Phi_{M'}"'] 
  & G_\mu \circ F_\mu(M')
\end{tikzcd}
\qquad\qquad
\begin{tikzcd}
N \arrow[r, "\Psi_N"] \arrow[d, "g"'] 
  & F_\mu \circ G_\mu (N)
      \arrow[d, "F_\mu \circ G_\mu(g)"] \\
N' \arrow[r, "\Psi_{N'}"'] 
  & F_\mu \circ G_\mu (N').
\end{tikzcd}
\]
The commutativity of these diagrams shows that the isomorphisms $\Phi_M$ and $\Psi_N$ are natural in $M$ and $N$, respectively, hence establishing the claimed equivalences of functors.
\end{proof}

\subsection{Charge shifts and spectral flow}\label{sec:spectral}
We fix the convention for spectral flow, since its sign depends on whether a twist is defined by pullback or inverse pullback. For $\mu\in L^*\subset L\otimes\mathbb C$, let $\sigma_\mu$ be the mode transformation induced by Li's operator \cite{L}
\[
\Delta(-\mu,z)=z^{-\mu_{(0)}}
\exp\left(\sum_{j\geq1}\frac{\mu_{(j)}}j(-z)^{-j}\right).
\]
We write $\sigma_\mu^*N$ for inverse pullback, so its fields are
$Y_N(\Delta(\mu,z)v,z)$. For a Heisenberg vector $h$ and a Heisenberg highest vector $v_\alpha$ of momentum $\alpha$, this gives
\[
Y_{\sigma_\mu^*N}(h,z)=Y_N(h,z)+(h,\mu)z^{-1},
\qquad
Y_{\sigma_\mu^*N}(v_\alpha,z)=z^{(\mu,\alpha)}Y_N(v_\alpha,z).
\]
The integral pairing with $L$ ensures that these are untwisted module fields.

\begin{theorem}\label{thm:sflow}\label{cor:sflow}
For $M\in\mathcal C_1$ and $\mu\in L^*$, there is a natural isomorphism
\[
\sigma_\mu^*F_0(M)\cong F_\mu(M).
\]
\end{theorem}
\begin{proof}
As vector spaces, identify $\pi_{-\lambda}$ with $\pi_{\mu-\lambda}$ by preserving the nonzero-mode creation operators and shifting the highest momentum. Fix the lattice-module cocycle on $\mu+L$ by this transport; a different coherent choice changes the map by a nonzero scalar on each momentum summand. The zero mode of $h$ changes by $(h,\mu)$. For $a_\alpha\otimes e^{-\alpha}\in V^2$, the zero-mode power in the lattice vertex operator changes by $z^{(-\alpha,\mu)}$; the creation and annihilation exponentials remain the same. These are precisely the two changes prescribed by $\Delta(\mu,z)$. Heisenberg descendants of these vectors span $V^2$, so the resulting map intertwines all fields. It is natural because it acts identically on $M$ and by the fixed transport on the Fock factor.
\end{proof}

In the applications to critical-level subregular $W$-algebras
in Section~\ref{sec:KSduality}, we take $L=\mathbb Z\lambda$
with $(\lambda,\lambda)=-1$ and identify the Heisenberg field
with $J=-\lambda$. For the charged generators $G^\pm$,
normalized by
\[
J(z)G^\pm(w)\sim\frac{\pm G^\pm(w)}{z-w},
\]
write $\sigma=\sigma_{-\lambda}$. This mode transformation
is induced by $\Delta(\lambda,z)=\Delta(-J,z)$ and satisfies
\[
\sigma(G^+_{(j)})=G^+_{(j-1)},\qquad
\sigma(G^-_{(j)})=G^-_{(j+1)},\qquad
\sigma(J_{(j)})=J_{(j)}+\delta_{j,0}.
\]
We use this integer spectral-flow convention throughout the paper,
in particular in Theorem~\ref{thm:Loewy}.

\subsection{Equivalence of categories with compatible Hamiltonian}
\label{sec:hamiltonian-equivalence}

Retain the notation of Section~\ref{sec:equiv}.
Fix an even Hamiltonian $H_1$ on $V^1$ preserving its
$L$-grading, and equip
\[
V^2=(V^1\otimes V_L)^T
\]
with the Hamiltonian
\begin{equation}\label{eq:induced-algebra-hamiltonian}
H_2=
\left.(H_1\otimes 1+1\otimes L_0^L)\right|_{V^2},
\end{equation}
where $L_0^L$ is the standard lattice Hamiltonian.
In particular, the Heisenberg generators of $V^2$
have $H_2$-weight $1$.
We refine the equivalence of
Theorem~\ref{thm:equivalence} by including a chosen
compatible Hamiltonian in the module data.

\begin{defn}\label{def:hamiltonian-categories}
Let $\mathcal C_{1,H}$ be the category whose objects
are pairs $(M,H_M)$, where $M\in\mathcal C_1$ and
$H_M$ is an even locally finite operator preserving
the $L$-grading and satisfying
\begin{equation}\label{eq:module-hamiltonian-compatibility}
[H_M,a_{(n)}]
=
(H_1a)_{(n)}-(n+1)a_{(n)},
\qquad a\in V^1,\quad n\in\mathbb Z.
\end{equation}
Morphisms
$f:(M,H_M)\to(M',H_{M'})$
are morphisms in $\mathcal C_1$ such that
\[
fH_M=H_{M'}f.
\]

For $\mu\in L^*$, define $\mathcal C_{2,H}^{\mu}$
similarly: its objects are pairs $(N,H_N)$ with
$N\in\mathcal C_2^\mu$ and an even locally finite
operator $H_N$ satisfying
\[
[H_N,a_{(n)}]
=
(H_2a)_{(n)}-(n+1)a_{(n)},
\qquad a\in V^2,\quad n\in\mathbb Z.
\]
Morphisms are morphisms in $\mathcal C_2^\mu$
intertwining the chosen Hamiltonians.
Write $\mathcal C_{1,H}^{\mathrm{ss}}$ and
$\mathcal C_{2,H}^{\mu,\mathrm{ss}}$ for the full
subcategories in which the chosen Hamiltonians
are semisimple.
\end{defn}

For a Heisenberg generator $h$, compatibility gives
\begin{equation}\label{eq:hamiltonian-heisenberg-modes}
[H_N,h(n)]=-n h(n).
\end{equation}
Thus $H_N$ commutes with the Heisenberg zero modes
and preserves the corresponding highest-weight spaces.
Both categories in Definition~\ref{def:hamiltonian-categories}
are abelian: kernels and cokernels in the underlying
categories inherit locally finite compatible Hamiltonians.
The same assertion holds for their semisimple
Hamiltonian subcategories.

Set
\[
q(\xi)=\frac12(\xi,\xi).
\]
The lattice Hamiltonian acts on the oscillator-degree
$k$ subspace of $\pi_\xi$ by $q(\xi)+k$.

\begin{lemma}\label{lem:hamiltonian-fock-decomposition}
Let $(N,H_N)\in\mathcal C_{2,H}^{\mu}$ and write
\[
N=\bigoplus_{\lambda\in L}
N_\lambda\otimes\pi_{\mu-\lambda}.
\]
There is a unique family of even locally finite
operators $D_\lambda$ on $N_\lambda$ such that
\begin{equation}\label{eq:hamiltonian-fock-splitting}
H_N\big|_{N_\lambda\otimes\pi_{\mu-\lambda}}
=
D_\lambda\otimes 1+
1\otimes L_0^L.
\end{equation}
Under the identification of $N_\lambda$ with the
Heisenberg-highest space
$N_\lambda\otimes e^{\mu-\lambda}$, one has
\begin{equation}\label{eq:hamiltonian-multiplicity}
D_\lambda
=
H_N\big|_{N_\lambda\otimes e^{\mu-\lambda}}
-q(\mu-\lambda).
\end{equation}
Moreover, $H_N$ is semisimple if and only if every
$D_\lambda$ is semisimple.
\end{lemma}

\begin{proof}
Equation~\eqref{eq:hamiltonian-heisenberg-modes}
shows that $H_N$ preserves each Heisenberg-highest
space. Define $D_\lambda$ by
\eqref{eq:hamiltonian-multiplicity}.
For $n_j>0$, compatibility gives
\[
\begin{split}
&H_N\bigl(
h_1(-n_1)\cdots h_r(-n_r)
(u\otimes e^{\mu-\lambda})\bigr)\\
&\quad=
h_1(-n_1)\cdots h_r(-n_r)
\left(
D_\lambda u\otimes e^{\mu-\lambda}
+
\left(q(\mu-\lambda)+\sum_{j=1}^r n_j\right)
u\otimes e^{\mu-\lambda}
\right).
\end{split}
\]
This proves \eqref{eq:hamiltonian-fock-splitting}
and uniqueness. Local finiteness follows by
restriction to the highest-weight space and
subtraction of a scalar.
Since $L_0^L$ is semisimple, the last assertion
follows from the same decomposition.
\end{proof}

We use the standard Hamiltonian $L_0^{L^-}$ on the
opposite lattice and the ghost Hamiltonian
$H_{\mathrm{gh}}$ for which the $b$- and $c$-fields
have weights $1$ and $0$, respectively.
The ghost vacuum has weight zero.
If $K_j$ are the diagonal isotropic currents, the
BRST current
\[
\sum_j :K_jc^j:
\]
has weight $1$. Consequently, its zero mode $Q$
commutes with the total Hamiltonian. This Hamiltonian
also preserves the relative subcomplex, since it
commutes with the relative zero-mode constraints.

\begin{prop}\label{prop:hamiltonian-functors}
The functors $F_\mu$ and $G_\mu$ lift to exact functors
\[
F_\mu^H:\mathcal C_{1,H}\longrightarrow
\mathcal C_{2,H}^{\mu},
\qquad
G_\mu^H:\mathcal C_{2,H}^{\mu}\longrightarrow
\mathcal C_{1,H},
\]
defined by
\begin{align}
F_\mu^H(M,H_M)
&=
\left(
(M\otimes V_{\mu+L})^T,\,
\left.(H_M\otimes1+1\otimes L_0^L)
\right|_{(M\otimes V_{\mu+L})^T}
\right),
\label{eq:hamiltonian-forward-functor}\\
G_\mu^H(N,H_N)
&=
\left(
H^{\mathrm{rel},0}
(N\otimes V_{-i\mu+L^-}),\,
\overline H_{\mathrm{tot}}
\right),
\label{eq:hamiltonian-inverse-functor}
\end{align}
where $\overline H_{\mathrm{tot}}$ is induced by
\[
H_{\mathrm{tot}}
=
H_N\otimes1\otimes1
+
1\otimes L_0^{L^-}\otimes1
+
1\otimes1\otimes H_{\mathrm{gh}}.
\]
Under the identification
$G_\mu(N)\cong\bigoplus_{\lambda\in L}N_\lambda$,
the induced Hamiltonian is
\begin{equation}\label{eq:inverse-hamiltonian-explicit}
H_{G_\mu^H(N)}\big|_{N_\lambda}=D_\lambda.
\end{equation}
Both functors preserve semisimplicity of the
chosen Hamiltonians.
\end{prop}

\begin{proof}
The operator in
\eqref{eq:hamiltonian-forward-functor}
preserves the invariant subspace because $H_M$
preserves charge and $L_0^L$ preserves lattice
momentum. Its compatibility with $H_2$ follows
from \eqref{eq:induced-algebra-hamiltonian}.
It is locally finite, being the sum of commuting
locally finite operators on the tensor product.

For the inverse construction, the preceding
observation gives
\[
[H_{\mathrm{tot}},Q]=0,
\]
so $H_{\mathrm{tot}}$ induces an operator on relative
cohomology. The reconstruction map
\[
a_\alpha\longmapsto
[a_\alpha\otimes e^{-\alpha}\otimes e^{i\alpha}]
\]
from Theorem~\ref{thm:isoalgebra} intertwines $H_1$
with the induced algebra Hamiltonian. Indeed,
\[
q(-\alpha)+q(i\alpha)=0,
\]
where the second quadratic form is computed in
the opposite lattice. The induced module operator
is therefore compatible with $H_1$.

For $u\in N_\lambda$, the corresponding
cohomology class is
\[
[u\otimes e^{\mu-\lambda}
\otimes e^{i\lambda-i\mu}].
\]
By Lemma~\ref{lem:hamiltonian-fock-decomposition},
the contribution of the first two factors is
$D_\lambda+q(\mu-\lambda)$, whereas the opposite
lattice contributes $-q(\mu-\lambda)$.
Thus
\[
\overline H_{\mathrm{tot}}
[u\otimes e^{\mu-\lambda}
\otimes e^{i\lambda-i\mu}]
=
[D_\lambda u\otimes e^{\mu-\lambda}
\otimes e^{i\lambda-i\mu}],
\]
which proves \eqref{eq:inverse-hamiltonian-explicit}
and local finiteness.

On morphisms, set
\[
F_\mu^H(f)=(f\otimes\operatorname{id})^T,
\qquad
G_\mu^H(g)=
H^{\mathrm{rel},0}(g\otimes\operatorname{id}).
\]
These maps intertwine the induced Hamiltonians.
Exactness follows from exactness of $F_\mu$ and
$G_\mu$ and the inherited Hamiltonians on kernels
and cokernels. Preservation of semisimplicity
follows from
\eqref{eq:hamiltonian-forward-functor} and
Lemma~\ref{lem:hamiltonian-fock-decomposition}.
\end{proof}

\begin{theorem}\label{thm:hamiltonian-equivalence}
For every $\mu\in L^*$, the functors $F_\mu^H$ and
$G_\mu^H$ are mutually quasi-inverse exact equivalences:
\[
\mathcal C_{1,H}\simeq\mathcal C_{2,H}^{\mu}.
\]
They restrict to exact equivalences
\[
\mathcal C_{1,H}^{\mathrm{ss}}
\simeq
\mathcal C_{2,H}^{\mu,\mathrm{ss}}.
\]
The natural isomorphisms are the reconstruction
maps of Theorem~\ref{thm:iso}.
\end{theorem}

\begin{proof}
It remains to check that the natural isomorphisms
$\Phi_M$ and $\Psi_N$ in Theorem~\ref{thm:iso}
intertwine the chosen Hamiltonians.

For $m_\lambda\in M_\lambda$,
\[
\Phi_M(m_\lambda)
=
[m_\lambda\otimes e^{\mu-\lambda}
\otimes e^{i\lambda-i\mu}].
\]
The two lattice ground-state energies cancel, so
\[
H_{G_\mu^HF_\mu^H(M)}\Phi_M(m_\lambda)
=
[H_Mm_\lambda\otimes e^{\mu-\lambda}
\otimes e^{i\lambda-i\mu}]
=
\Phi_M(H_Mm_\lambda).
\]

Next, for $u\in N_\lambda$ and
$s\in\pi_{\mu-\lambda}$,
\[
\Psi_N(u\otimes s)
=
[u\otimes e^{\mu-\lambda}
\otimes e^{i\lambda-i\mu}]\otimes s.
\]
Equations~\eqref{eq:hamiltonian-fock-splitting}
and~\eqref{eq:inverse-hamiltonian-explicit} give
\[
H_{F_\mu^HG_\mu^H(N)}\Psi_N(u\otimes s)
=
\Psi_N(D_\lambda u\otimes s+
u\otimes L_0^Ls)
=
\Psi_N(H_N(u\otimes s)).
\]
Thus both reconstruction maps are morphisms in
the refined categories. Their naturality follows
from the naturality of the underlying reconstruction
maps. Proposition~\ref{prop:hamiltonian-functors}
gives the restriction to semisimple Hamiltonians.
\end{proof}

\begin{cor}\label{cor:hamiltonian-weight-restriction}
The equivalence in
Theorem~\ref{thm:hamiltonian-equivalence}
restricts to the full subcategories satisfying
the following conditions:
\begin{enumerate}
\item
every joint charge and generalized Hamiltonian
weight space is finite-dimensional;
\item
at each fixed charge, the real parts of the
Hamiltonian weights are bounded below.
\end{enumerate}
The same assertion holds with semisimple
Hamiltonians.
\end{cor}

\begin{proof}
Let $M_{\lambda,\Delta}$ be the generalized
$H_M$-eigenspace of eigenvalue $\Delta$ in $M_\lambda$,
and let $\pi_{\mu-\lambda}[k]$ denote the
oscillator-degree $k$ subspace.
The generalized weight space of energy $E$ in
the summand of momentum $\mu-\lambda$ is
\[
\bigl(F_\mu^H(M)\bigr)_{\lambda,E}
=
\bigoplus_{k\geq0}
M_{\lambda,E-q(\mu-\lambda)-k}
\otimes\pi_{\mu-\lambda}[k].
\]
If $\operatorname{Re}\Delta\geq b_\lambda$ on
$M_\lambda$, only finitely many $k$ contribute
for fixed $E$, and each oscillator space is
finite-dimensional. Moreover, the real parts of
the resulting energies are bounded below by
$b_\lambda+\operatorname{Re}q(\mu-\lambda)$.

Conversely, the Heisenberg-highest space is a
Hamiltonian-stable subspace of the fixed-momentum
summand. Subtracting $q(\mu-\lambda)$ from its
Hamiltonian preserves finite-dimensional
generalized weight spaces and lower boundedness.
The assertion follows from
\eqref{eq:inverse-hamiltonian-explicit}.
\end{proof}

\section{Kazama-Suzuki duality and quasi-classical limit}\label{sec:limits}

\subsection{Kazama-Suzuki duality at the generic level} \label{sec:KSgen}

 Let $\mathfrak{g}=\mathfrak{sl}_n$ or $\mathfrak{g}=\mathfrak{so}_{2n+1}$, and let $f=f_{\mathrm{sub}}$ be a subregular nilpotent element. 
The subregular $W$-algebra $W^k(\mathfrak{g},f_{\mathrm{sub}})$ is strongly and freely generated by a Heisenberg field $J$, two even fields $G^\pm$, and additional even fields $W_i$ as follows:
\begin{itemize}
    \item for $\mathfrak{g}=\mathfrak{sl}_n$, the fields $G^\pm$ have conformal weight $\frac{n}{2}$, and there is one field
    \[
    W_2,W_3,\dots,W_{n-1}
    \]
    of each conformal weight $2,3,\dots,n-1$;

    \item for $\mathfrak{g}=\mathfrak{so}_{2n+1}$, the fields $G^\pm$ have conformal weight $n$, and there is one field
    \[
    W_2,W_4,\dots,W_{2n-2}
    \]
    of each conformal weight $2,4,\dots,2n-2$.
\end{itemize}
All generators are neutral with respect to the Heisenberg field $J$, except for $G^\pm$, which have $J$-charge $\pm1$. 
This determines the normalization of $J$. 
Its level is
\[
-\frac{n-(n-1)(k+n)}{n}
\]
for $\mathfrak{g}=\mathfrak{sl}_n$ \cite[Lemma 3.4]{CL1}, and
\[
k+2n-2
\]
for $\mathfrak{g}=\mathfrak{so}_{2n+1}$ \cite[Section 3.3 (3)]{CL2}. 
In particular, at the critical level $k+n=0$ for $\mathfrak{sl}_n$, respectively $k+2n-1=0$ for $\mathfrak{so}_{2n+1}$, the field $J$ has level $-1$.

For $\gg=\mathfrak{sl}_{n|1}$ or $\mathfrak{osp}_{2|2n}$, the principal $W$-superalgebra $W^k(\mathfrak{g})$ is strongly generated by a Heisenberg field $J$, two odd fields $G^\pm$, and additional even fields $W_i$ as follows:
\begin{itemize}
    \item for $\mathfrak{g}=\mathfrak{sl}_{n|1}$, the  fields $G^\pm$ have conformal weight $\frac{n+1}{2}$, and there is one field
    \[
    W_2,W_3,\dots,W_n
    \]
of each  conformal weights $2,3,\dots,n$;

    \item for $\mathfrak{g}=\mathfrak{osp}_{2|2n}$, the  fields $G^\pm$ have conformal weight $n+\frac12$, and there is one  field
    \[
    W_2,W_4,\dots,W_{2n}
    \]
of each conformal weights $2,4,\dots,2n$.
\end{itemize}
Again, all generators are neutral with respect to $J$, except for $G^\pm$, which have $J$-charge $\pm1$. 
This fixes the normalization of $J$, whose level is
\[
1+\frac{n(k+n-1)}{1-n}
\]
for $\mathfrak{g}=\mathfrak{sl}_{n|1}$ \cite{CL1}. 
In particular, when $n=2$, the algebra $W^k(\mathfrak{sl}_{2|1})$ coincides with the $N=2$ superconformal vertex algebra.

We recall a Kazama--Suzuki type coset duality for 
\[
(\mathfrak g_1,\mathfrak g_2)
=
(\mathfrak {sl}_{n},\mathfrak {sl}_{n|1})
\quad \text{or} \quad
(\mathfrak{so}_{2n+1},\mathfrak {osp}_{2|2n}).
\]
Let $h_i^\vee$ be the dual Coxeter number of $\mathfrak g_i$, and let $r$ denote the lacity of $\mathfrak g_1$. The relation between the corresponding levels $k_1$ and $k_2$ is given by
\begin{equation}\label{rel:level}
r\,(k_1+h_1^\vee)(k_2+h_2^\vee)=1.
\end{equation}
\begin{theorem}\cite{CGN} \label{thm:genericlevel}
 \label{thm:KZisom}
Let $(\mathfrak g_1,\mathfrak g_2)=(\mathfrak {sl}_{n},\mathfrak {sl}_{n|1})$ or $(\mathfrak{so}_{2n+1},\mathfrak {osp}_{2|2n})$. 
Let $k_1 \in \mathbb C \setminus \{-h_1^\vee\}$ and 
$k_2 \in \mathbb C \setminus \{-h_2^\vee\}$ be generic levels satisfying \eqref{rel:level}. 
Then there are mutually inverse Kazama-Suzuki type coset isomorphisms
\begin{align*}
W^{k_1}(\mathfrak g_1,f_{\mathrm{sub}})
&\cong 
\mathrm{Com}\!\left(\pi_{\widetilde H_2},\, 
W^{k_2}(\mathfrak g_2)\otimes V_{\mathbb Z \sqrt{-1}}\right),\\
W^{k_2}(\mathfrak g_2)
&\cong 
\mathrm{Com}\!\left(\pi_{\widetilde H_1},\, 
W^{k_1}(\mathfrak g_1,f_{\mathrm{sub}})\otimes V_{\mathbb Z}\right),
\end{align*}
and also for the corresponding simple quotients
\begin{align*}
W_{k_1}(\mathfrak g_1,f_{\mathrm{sub}})
&\cong 
\mathrm{Com}\!\left(\pi_{\widetilde H_2},\, 
W_{k_2}(\mathfrak g_2)\otimes V_{\mathbb Z \sqrt{-1}}\right),\\
W_{k_2}(\mathfrak g_2)
&\cong 
\mathrm{Com}\!\left(\pi_{\widetilde H_1},\, 
W_{k_1}(\mathfrak g_1,f_{\mathrm{sub}})\otimes V_{\mathbb Z}\right).
\end{align*}
\end{theorem}

Let \(\pi\) and \(\pi^{-}\) denote the Heisenberg vertex algebras of levels \(1\) and \(-1\), respectively. 
For any Heisenberg element \(H\), we write \(\pi^H\) for the Heisenberg vertex algebra generated by \(H\).
The $W$-superalgebras $W^{k_1}(\mathfrak g_1,f_{\mathrm{sub}})$
and $W^{k_2}(\mathfrak g_2)$ contain Heisenberg vertex
subalgebras $\pi^{H^1}$ and $\pi^{H^2}$, respectively,
with corresponding pairs of levels
\[
\begin{cases}
\displaystyle
\left(
-\frac{n-(n-1)(k_1+n)}{n},
\;
1+\frac{n(k_2+n-1)}{1-n}
\right),
&
(\mathfrak g_1,\mathfrak g_2)
=(\mathfrak{sl}_n,\mathfrak{sl}_{n|1}),
\\[8pt]
\displaystyle
\left(k_1+2n-2,\;1-2(k_2+n)\right),
&
(\mathfrak g_1,\mathfrak g_2)
=(\mathfrak{so}_{2n+1},\mathfrak{osp}_{2|2n}).
\end{cases}
\]
By \cite{CGNS}, one has the decomposition
\begin{equation} \label{eq:decom}
 V_{\mathbb Z}\otimes \pi^{-}
\cong
\bigoplus_{\lambda \in \mathbb Z}
\pi_{\lambda}^{\sqrt{-1}H^1}
\otimes
\pi_{-\lambda}^{H^2},
\end{equation}
as a
$\pi^{\sqrt{-1}H^1}\otimes \pi^{H^2}$-module.
Since $W^{k_1}(\gg_1,f_{\text{sub}}) \otimes \pi^{\sqrt{-1} H^1}$ contains a degenerate Heisenberg element, hence one can consider the relative semi-infinite cohomology of $W^{k_1}(\mathfrak g_1,f_{\mathrm{sub}})
\otimes  V_{\mathbb Z}
\otimes \pi^{-}$ with respect to $\mathfrak{gl}_1$.
Using the decomposition (\ref{eq:decom}) and Lemma \ref{lem:vanish},  we deduce the isomorphisms 
\[
H^{\mathrm{rel},\, p}\!\left(
W^{k_1}(\mathfrak g_1,f_{\mathrm{sub}})
\otimes V_{\mathbb Z}
\otimes \pi^{-}
\right)
\cong
\delta_{p,0}\, W^{k_2}(\mathfrak g_2)\]
of vertex algebra modules over $\pi^{\sqrt{-1} H^1}\otimes \pi^{H^2}$.
Similarly by decomposing $V_{\sqrt{-1}\mathbb Z} \otimes \pi^+$, one can obtain an isomorphism for reverse direction
and hence reformulate the Kazama--Suzuki type
coset isomorphism in terms of relative semi-infinite cohomology.

\begin{theorem} \cite{CGNS, N} \label{thm:glevel}
There are isomorphisms 
\begin{align} \label{eq:glevel}
H^{\mathrm{rel},\, p}\!\left(
W^{k_1}(\mathfrak g_1,f_{\mathrm{sub}})
\otimes  V_{\mathbb Z}
\otimes \pi^{-}
\right)
&\cong
\delta_{p,0}\, W^{k_2}(\mathfrak g_2),
\\  \label{eq:glevel2}
H^{\mathrm{rel},\, p}\!\left(
W^{k_2}(\mathfrak g_2)
\otimes  V_{\sqrt{-1}\mathbb Z}
\otimes \pi^{+}
\right)
&\cong
\delta_{p,0}\, W^{k_1}(\mathfrak g_1,f_{\mathrm{sub}}).
\end{align}

\end{theorem}

We use the standard notation  for the infinite $q$-Pochhammer symbol $(a;q)_{\infty}=\prod_{n=0}^{\infty} (1-aq^n)$.
Consider the vertex superalgebra $\cC:=W^{k_1}(\mathfrak g_1,f_{\mathrm{sub}}) \otimes V_{\mathbb Z} \otimes \pi^- \otimes \mathcal E$. We identify $V_{\mathbb Z} $ with the $bc$-system with the field generated by $\phi$ and $\psi$.
Let 
\[
L=L^1+L_{V_{\mathbb Z}}+L_{\pi^-}+L_{\mathcal E}\in \cC
\]
be the conformal vector of $\cC$, where $L^1$ is the standard conformal vector in $W^{k_1}(\mathfrak g_1,f_{\mathrm{sub}})$,  $L_{V_{\mathbb Z}}=\frac{1}{2} (\partial \psi \phi-\psi \partial \phi)$,
$L_{\pi^-}=-\frac12 h_{(-1)}h$, $L_{\mathcal E}=-b\partial c $ are the conformal vectors in $V_{\mathbb Z} $, $\pi^-$, and $\mathcal E$, respectively.
We derive the following identities from
Theorem~\ref{thm:glevel} by comparing ordinary graded
Euler characters of the relative BRST complexes.
The alternating sign refers to cohomological degree,
independently of vertex-superalgebra parity.

\begin{cor}\label{cor:comiden}
The following combinatorial identities hold for $|q|<1$
\begin{align} \label{eq:comb1}
\left.
\frac{
(-y^{-1}zq^{\frac12};q)_\infty
(-yz^{-1}q^{\frac12};q)_\infty
}{
\displaystyle
(zq^{\frac n2};q)_\infty
\,
(z^{-1}q^{\frac n2};q)_\infty
}
\right|_{z^0}&=\frac{(-yq^{\frac{n+1}2};q)_\infty
(-y^{-1}q^{\frac{n+1}2};q)_\infty }{(q;q)_\infty \, (q^n;q)_{\infty} },\\  \label{eq:comb2}
\left.
\frac{
(-y^{-1}zq^{\frac12};q)_\infty
(-yz^{-1}q^{\frac12};q)_\infty
}{
\displaystyle
(zq^{n};q)_\infty
\,
(z^{-1}q^{n};q)_\infty
}
\right|_{z^0}&=\frac{(-yq^{n+\frac{1}2};q)_\infty
(-y^{-1}q^{n+\frac{1}2};q)_\infty }{(q;q)_\infty \, (q^{2n};q)_{\infty} },
\end{align}
where $|_{z^0}$ refers to the $z^0$-coefficient.
\end{cor}

\begin{proof}
Choose generic levels satisfying the duality relation.
We treat the two cases simultaneously by setting
\[
(d,I)=
\begin{cases}
\bigl(n,\{1,2,\ldots,n-1\}\bigr),
&(\mathfrak g_1,\mathfrak g_2)
 =(\mathfrak{sl}_n,\mathfrak{sl}_{n|1}),\\[2mm]
\bigl(2n,\{1\}\cup\{2j:1\leq j\leq n-1\}\bigr),
&(\mathfrak g_1,\mathfrak g_2)
 =(\mathfrak{so}_{2n+1},\mathfrak{osp}_{2|2n}).
\end{cases}
\]
Thus the neutral strong generators of
$W^{k_1}(\mathfrak g_1,f_{\mathrm{sub}})$ have weights
in $I$, and its two charged even generators have
weight $d/2$.

Consider the BRST complex
\[
C=
W^{k_1}(\mathfrak g_1,f_{\mathrm{sub}})
\otimes V_{\mathbb Z}\otimes\pi^-\otimes\mathcal E,
\]
equipped with the sum of the standard conformal vectors
of its four factors. 
Let $H^3$ denote the auxiliary Heisenberg current
generating the factor denoted by
$\pi^{\sqrt{-1}H^1}$ in \eqref{eq:decom}.
The currents $H^1$ and $H^3$ have opposite levels,
so $K=H^1+H^3$ is isotropic.
The residual current $H^2$ commutes with $K$.
Hence the BRST differential preserves both $L_0$
and $H^2_{(0)}$.
We use the ordinary graded Euler character
\[
\chi_{\mathrm{Eul}}(C^{\mathrm{rel}};q,y)
=
\sum_{p\in\mathbb Z}(-1)^p
\operatorname{tr}_{C^{\mathrm{rel},p}}
\bigl(q^{L_0}y^{H^2_{(0)}}\bigr).
\]
Every fixed conformal-weight subspace is finite-dimensional,
and only finitely many cohomological degrees occur at
that weight. Hence the Euler--Poincar\'e identity applies
coefficientwise.

The condition $K_0=0$ is imposed by taking the
$z^0$-coefficient, where $z$ records $K_0$-charge.
With the charge conventions of~\eqref{eq:decom}, the
charged even generators contribute
\[
\frac{1}
{(zq^{d/2};q)_\infty(z^{-1}q^{d/2};q)_\infty},
\]
and the two lattice fermions contribute
\[
(-y^{-1}zq^{1/2};q)_\infty
(-yz^{-1}q^{1/2};q)_\infty.
\]
The auxiliary Heisenberg vacuum algebra contributes
$(q;q)_\infty^{-1}$.
The relative ghost condition excludes $c_0$; the remaining
ghost creators have weights $m\geq1$ and cohomological
degrees $1$ and $-1$. They are neutral under both charge
operators, so their Euler character is
\[
\prod_{m\geq1}(1-q^m)^2=(q;q)_\infty^2.
\]
Consequently,
\[
\chi_{\mathrm{Eul}}(C^{\mathrm{rel}};q,y)
=
\frac{(q;q)_\infty}
     {\displaystyle\prod_{i\in I}(q^i;q)_\infty}
\left.
\frac{
(-y^{-1}zq^{1/2};q)_\infty
(-yz^{-1}q^{1/2};q)_\infty
}{
(zq^{d/2};q)_\infty
(z^{-1}q^{d/2};q)_\infty
}
\right|_{z^0}.
\]

By Theorem~\ref{thm:glevel}, the relative cohomology
vanishes outside degree zero and its degree-zero part
is $W^{k_2}(\mathfrak g_2)$.
The conformal Kazama--Suzuki realization identifies the
induced conformal grading and residual Heisenberg charge
with the standard ones on this algebra.
Its neutral PBW generators have weights in
$I\cup\{d\}$, and its two odd generators have weight
$(d+1)/2$ and charges $\pm1$. Therefore
\[
\chi_{\mathrm{Eul}}(C^{\mathrm{rel}};q,y)
=
\operatorname{tr}_{W^{k_2}(\mathfrak g_2)}
\bigl(q^{L_0}y^{H^2_{(0)}}\bigr)
=
\frac{
(-yq^{(d+1)/2};q)_\infty
(-y^{-1}q^{(d+1)/2};q)_\infty
}{
(q^d;q)_\infty
\displaystyle\prod_{i\in I}(q^i;q)_\infty
}.
\]
Comparing the two expressions and cancelling the common
neutral factors gives
\[
\left.
\frac{
(-y^{-1}zq^{1/2};q)_\infty
(-yz^{-1}q^{1/2};q)_\infty
}{
(zq^{d/2};q)_\infty
(z^{-1}q^{d/2};q)_\infty
}
\right|_{z^0}
=
\frac{
(-yq^{(d+1)/2};q)_\infty
(-y^{-1}q^{(d+1)/2};q)_\infty
}{
(q;q)_\infty(q^d;q)_\infty
}.
\]
Taking $d=n$ and $d=2n$ proves
\eqref{eq:comb1} and \eqref{eq:comb2}, respectively.
\end{proof}

\subsection{Rescaled large-level limits}
Let $\mathcal V_\epsilon$ be a vertex superalgebra over $\mathbb C[\epsilon]$, or over a localization consisting of functions regular at $\epsilon=0$. Assume it has a PBW basis in finitely many homogeneous strong generators $X^\alpha$, and that all singular OPE coefficients are divisible by $\epsilon$. The specialization $\mathcal V_\epsilon/\epsilon\mathcal V_\epsilon$ is then commutative. A different rescaling can retain noncommutative OPEs.

Suppose the generators are charge graded, with neutral generators indexed by $S^0$ and charged generators by $S^{\mathrm{ch}}$. Make the formal base change $\epsilon=\sigma^2$, and inside the localization at $\sigma$ set
\[
Y^\alpha=\sigma^{-1}X^\alpha\quad(\alpha\in S^{\mathrm{ch}}),
\qquad Y^\alpha=X^\alpha\quad(\alpha\in S^0).
\]
Let $\mathcal V_\sigma$ be their $\mathbb C[\sigma]$-span under normally ordered products and derivatives. The divisibility assumption makes this span closed under all products. Its ordered rescaled PBW monomials are a free basis, so specialization introduces no additional linear relations. We call
$V^\infty=\mathcal V_\sigma/\sigma\mathcal V_\sigma$ the (rescaled) large-level limit. This is not, in general, the commutative specialization of the unrescaled family.

\begin{prop}\label{prop:OPE}
Suppose
$[X^\alpha{}_{\lambda}X^\beta]=\epsilon R_{\alpha\beta}(X,\epsilon,\lambda)$.
All neutral generators become central in $V^\infty$. For two charged generators, their limiting bracket is obtained from $R_{\alpha\beta}$ by putting $\epsilon=0$, setting every charged $X$ to zero, and replacing the remaining neutral $X$ by $Y$. In particular the bracket vanishes unless the two charges sum to zero.
\end{prop}
\begin{proof}
For two charged fields the factor $\sigma^{-2}$ cancels $\epsilon=\sigma^2$. Substituting $X^{\mathrm{ch}}=\sigma Y^{\mathrm{ch}}$ then removes every monomial containing a charged generator. With one or two neutral inputs an additional factor $\sigma$ or $\sigma^2$ remains, so the bracket vanishes at $\sigma=0$. Charge conservation gives the final assertion. This is the rescaling argument of \cite[Proposition~3.2]{CL1}, with the charged and neutral generators scaled differently.
\end{proof}
The applications below use the regular PBW families of \cite{CL1,CL2}. Their neutral limiting fields generate a commutative differential subalgebra; quotienting by its augmentation is the additional operation producing the centerless free fields. We write $W^\infty$ for the limit retaining these fields and $W_\infty$ for its indicated centerless quotient. The choice of regular family and rescaling is part of this notation.

\begin{example} \label{ex:affine}
Let $\gg$ be a Lie superalgebra with the normalized invariant supersymmetric bilinear form $\kappa$. For any complex number $l$,  we denote by $V^{l}(\gg)$ the corresponding universal affine vertex algebra at level $l$ with strong generating set $\{J^\xi\}_{\xi\in \gg}$ such that
\[
J^{\xi}(z)\, J^{\eta}(w)
\;\sim\;
\frac{l\,\kappa(\xi,\eta)}{(z-w)^2}
+\frac{J^{[\xi,\eta]}(w)}{z-w},
\qquad
\text{for any } \xi,\eta \in \mathfrak{g}.
\]
Fix a nonzero $k\in \mathbb C$ and let $l=\epsilon^{-1} k$. We recall the regular family $V_{\epsilon}^k(\gg)$ by scaling the strong generators by $\epsilon$ as in \cite{CL1}, and it is the vertex superalgebra strongly generated by $\{J^\xi_\epsilon\}_{\xi \in \gg}$ such that
\[
J^{\xi}_{\epsilon}(z)\, J^{\eta}_{\epsilon}(w)
\;\sim\;
\frac{\epsilon k\,\kappa(\xi,\eta)}{(z-w)^2}
+\frac{\epsilon J^{[\xi,\eta]}_{\epsilon}(w)}{z-w},
\qquad
\text{for any } \xi,\eta \in \mathfrak{g}.
\]
Let $\{X^\alpha\}_{\alpha\in S}$ be a collection of strong generators of $V_{\epsilon}^k(\gg)$ with a finite index set  $S=S^+\cup S^0\cup S^-$, which coincides with the generalized root system of $\gg$. 
Set $Y^\alpha=\sigma^{-1} X^{\alpha}$ for $\alpha\in S^\pm$ and $Y^\alpha=X^\alpha$ for $\alpha\in S^0$ and denote by $V^k_{\sigma}(\gg)$ the vertex superalgebra generated by $\{Y^\alpha\}_{\alpha\in S}$. We obtain  $V^{\infty}(\gg):=V^k_{\sigma}(\gg)/\sigma V^k_{\sigma}(\gg)$ the large level limit of the affine vertex algebra.
\end{example}

Below, we describe explicitly the large level limit of the affine vertex algebra associated with $\mathfrak{sl}_2$.
Let $e,f,h$ be the canonical basis of $\mathfrak {sl}_2$, satisfying that
\begin{equation}\label{eq:sl2}
[e,f]=h, \quad [h,e]=2e,\quad [h,f]=-2f.
\end{equation}
Let $V^k(\mathfrak{sl}_2)$ be the universal affine vertex algebra of $\mathfrak {sl}_2$ at the level $k$. We define
\begin{equation}
e^{(k)}:=\frac{1}{\sqrt k} e, \quad f^{(k)}:=\frac{1}{\sqrt k} f,\quad h^{(k)}:=\frac{1}{k} h.
\end{equation}
By taking the limit $k\to \infty$, we obtain the large limit $\mathfrak {sl}_2^{\infty}$, which is a three-dimensional Lie algebra with a one-dimension center, spanned by $e^{\infty},f^{\infty},h^{\infty}$ and satisfying the nontrivial relation 
\[ [e^{\infty},f^{\infty}]=h^{\infty}. \]
The corresponding affine Lie algebra is defined to be
\[\widehat{\mathfrak{sl}}_2^\infty:=\mathfrak{sl}_2^\infty[t,t^{-1}] \oplus \mathbb C K,
\]
with the nontrivial Lie bracket 
\[
[e^{\infty}\otimes t^m,f^{\infty}\otimes t^n]=h^{\infty}\otimes  t^{m+n}+m\delta_{m+n,0} K.
\]
The large-level limit of $V^k(\mathfrak{sl}_2)$, denoted by $V^{\infty}(\mathfrak{sl}_2)$, is the vacuum representation of the affine Lie algebra $\widehat{\mathfrak{sl}}_2^{\,\infty}$ at level~$1$. Namely,
\[
V^{\infty}(\mathfrak{sl}_2)
\cong U(\widehat{\mathfrak{sl}}_2^{\,\infty})
  \otimes_{U(\mathfrak{sl}_2^{\infty}[t] \oplus \mathbb{C}K)} \mathbb{C},
\]
where $\mathfrak{sl}_2^{\infty}[t]$ acts trivially on $\mathbb{C}$, and $K$ acts as the identity. We note that the vertex algebra $V^{\infty}(\mathfrak {sl}_2)$ contains a big center, and its maximal ideals are not unique. We denote by $V_{\infty}(\mathfrak {sl}_2)$ its  quotient vertex algebra by the augmented ideal generated by the center, which is a free field algebra generated by two elements $e^\infty$ and $f^\infty$ with the nontrivial OPE
\[
e^\infty(z) f^{\infty}(w)\sim \frac{1}{(z-w)^2}.
\]
In particular, the quotient vertex algebra $V_{\infty}(\mathfrak {sl}_2)$ is simple.

\begin{example}
Let $\mathcal{W}^{k}_{\epsilon}(\mathfrak{g},f)$ be the regular family of $\mathcal{W}$-algebras defined in \cite{CL1}, together with a collection of strong generators
$
\{X^\alpha\}_{\alpha \in S},
$
indexed by a set $S$.
Assume that there exists a Heisenberg element $X^{\alpha_0}$ for some $\alpha_0 \in S$. Moreover, suppose that the index set $S$ admits a decomposition
$
S = S^{+} \sqcup S^{0} \sqcup S^{-},
$
which is compatible with the Heisenberg grading induced by $X^{\alpha_0}$.
Set $Y^\alpha = \sigma^{-1} X^\alpha$ for $\alpha \in S^\pm$, and $Y^\alpha = X^\alpha$ for $\alpha \in S^0$. 
Similarly, we denote by $W_{\sigma}^k(\mathfrak{g},f)$ the vertex superalgebra generated by $\{Y^\alpha\}_{\alpha \in S}$. 
We define 
$
W^{\infty}(\mathfrak{g},f) := W_{\sigma}^k(\mathfrak{g},f) / \sigma W_{\sigma}^k(\mathfrak{g},f)
$
to be the large level limit of the regular family, and denote by $W_{\infty}(\mathfrak{g},f)$ its quotient by the augmented ideal generated by the center. When $f$ is a principal nilpotent element, we simplify the notation $W^k(\mathfrak{g}, f)$ $(\text{resp.}\;  W_k(\mathfrak{g}, f))$ to $W^k(\mathfrak{g})$ $(\text{resp.}\; W_k(\mathfrak{g}))$.

\end{example}

For $W^k(\mathfrak{sl}_{2|1})$, consider the rescaled fields
\[
J^{(k)}=k^{-1}J,\qquad
L^{(k)}=k^{-1}L,\qquad
G^{\pm,(k)}=k^{-1/2}G^\pm.
\]
Their large-level limit $W^\infty(\mathfrak{sl}_{2|1})$
is freely strongly generated by the central even fields
$J^\infty,L^\infty$ and the odd fields $G^{\pm,\infty}$,
with the only nontrivial generating OPE, up to
skew-symmetry,
\begin{equation}\label{eq:GOPE}
G^{+,\infty}(z)G^{-,\infty}(w)\sim
-\frac{2}{(z-w)^3}
+\frac{J^\infty(w)}{(z-w)^2}
+\frac{L^\infty(w)+\frac12\partial J^\infty(w)}{z-w}.
\end{equation}
Quotienting by the ideal generated by $J^\infty,L^\infty$
gives the simple free-field algebra
$W_\infty(\mathfrak{sl}_{2|1})$, with
\begin{equation}\label{eq:GsOPE}
G^{+,\infty}(z)G^{-,\infty}(w)
\sim-\frac{2}{(z-w)^3}.
\end{equation}

\subsection{Application to the $N=2$ superconformal algebra}
\label{subsec:N2-application}

We now consider the reverse-limit construction.
Write $e=e^\infty$, $f=f^\infty$ for the generators of
$D=V_\infty(\mathfrak{sl}_2)$, and realize
$V_{\mathbb Z}$ as the charged free-fermion algebra
generated by $b,c$. Their nontrivial generating OPEs are
\[
e(z)f(w)\sim\frac{1}{(z-w)^2},
\qquad
b(z)c(w)\sim\frac{1}{z-w}.
\]
Let
\[
B=(D\otimes V_{\mathbb Z})^{\mathbb C^*},
\qquad
t\cdot(e,f,b,c)=(te,t^{-1}f,tb,t^{-1}c).
\]
The following fields belong to $B$:
\begin{equation}\label{eq:N2-orbifold-fields}
\begin{aligned}
G^+&=:ec:,\qquad
G^-=:fb:,\qquad J=-:bc:,\\
L&=:ef:
+\frac12\bigl(:\partial b\,c:-:b\,\partial c:\bigr).
\end{aligned}
\end{equation}
Let $\mathcal U_3$ denote the universal $N=2$
superconformal vertex algebra of central charge $3$.

\begin{theorem}\label{thm:N2-c3}
The fields \eqref{eq:N2-orbifold-fields} define an
isomorphism of conformal vertex superalgebras
\[
\mathcal U_3
\xrightarrow{\;\sim\;}
\bigl(V_\infty(\mathfrak{sl}_2)
      \otimes V_{\mathbb Z}\bigr)^{\mathbb C^*}.
\]
In particular, $\mathcal U_3$ is simple.
\end{theorem}

\begin{proof}
Wick contraction gives the $N=2$ OPEs at central charge
$3$, in the normalization
\[
G^+(z)G^-(w)\sim
\frac{1}{(z-w)^3}
+\frac{J(w)}{(z-w)^2}
+\frac{L(w)+\frac12\partial J(w)}{z-w}.
\]
Thus we obtain a conformal homomorphism
$\varphi:\mathcal U_3\to B$.

Let $A=\operatorname{im}\varphi$. The fields
\[
L_F=\frac12:JJ:
=\frac12\bigl(:\partial b\,c:-:b\,\partial c:\bigr),
\qquad
L_D=L-L_F=:ef:
\]
belong to $A$, and their zero modes act as translation
on the respective tensor factors. Hence, for $i,j\geq0$,
\begin{align*}
P_{i,j}
&=(L_D)_{(0)}^i(L_F)_{(0)}^jG^+
 =:\partial^i e\,\partial^j c:\in A,\\
Q_{i,j}
&=(L_D)_{(0)}^i(L_F)_{(0)}^jG^-
 =:\partial^i f\,\partial^j b:\in A.
\end{align*}
All fields $:\partial^i b\,\partial^j c:$ lie in the
lattice-charge-zero subalgebra of $V_{\mathbb Z}$,
which is the Heisenberg algebra generated by $J$.
Moreover,
\[
(P_{i,0})_{(0)}Q_{j,0}
=
:\partial^i e\,\partial^j f:
+(-1)^i:\partial^{i+j+1}c\,b:,
\]
so $:\partial^i e\,\partial^j f:$ also belongs to $A$.

These four families of quadratic fields generate $B$.
Indeed, taking torus invariants commutes with passing
to the associated graded algebra for the free-field
length filtration. Every invariant monomial has equally
many factors of charges $1$ and $-1$, and thus factors
into quadratic invariants. Since their lifts belong
to $A$, induction on the filtration degree gives
$A=B$.

Finally, \eqref{eq:comb1} with $n=2$, $y=1$, and the
PBW theorem give
\begin{align*}
\operatorname{ch}B
&=
\left.
\frac{(-zq^{1/2};q)_\infty
      (-z^{-1}q^{1/2};q)_\infty}
     {(zq;q)_\infty(z^{-1}q;q)_\infty}
\right|_{z^0}\\
&=
\frac{(-q^{3/2};q)_\infty^2}
     {(q;q)_\infty(q^2;q)_\infty}
=\operatorname{ch}\mathcal U_3.
\end{align*}
Here characters are ordinary graded characters.
Since $\varphi$ is surjective and the homogeneous
components are finite-dimensional, it is an isomorphism.

The ambient algebra $D\otimes V_{\mathbb Z}$ is simple,
and its $\mathbb C^*$-invariants equal its
$U(1)$-invariants. The fixed-point simplicity argument
of \cite{DLM}, valid also for vertex superalgebras,
therefore shows that $B$, and hence $\mathcal U_3$,
is simple.
\end{proof}

\section{Large and critical level duality}

\subsection{Duality theorem}\label{sec:KSduality}
We first record the Hamiltonian conditions used to compare particular representations with familiar weight-module settings.

\begin{defn}\label{defcatwtmod}
Let $V$ be a vertex algebra with Hamiltonian $H$ and
Heisenberg field $J$ of $H$-weight $1$.
A Hamiltonian weight module is a pair $(M,H_M)$, where
$M$ is a finitely generated $V$-module on which $J_0$
acts semisimply, and $H_M$ is an even locally finite
operator satisfying
\[
[H_M,a_{(n)}]
=
(Ha)_{(n)}-(n+1)a_{(n)},
\qquad a\in V,\quad n\in\mathbb Z.
\]
In particular, $[H_M,J_0]=0$. We require
\[
M=\bigoplus_{\lambda,\Delta\in\mathbb C}
M_{\lambda,\Delta},
\qquad
\dim M_{\lambda,\Delta}<\infty,
\]
where
\[
M_{\lambda,\Delta}
=
\left\{
m\in M \;\middle|\;
J_0m=\lambda m,\quad
(H_M-\Delta)^r m=0
\text{ for some }r\geq1
\right\}.
\]
For each fixed $\lambda$, we further require
\[
\inf\left\{
\operatorname{Re}\Delta
\;\middle|\;
M_{\lambda,\Delta}\neq0
\right\}>-\infty.
\]

Denote the category of these pairs by
$V$-\textup{wtmod}.
A morphism
$f:(M,H_M)\to(M',H_{M'})$
is a $V$-module homomorphism satisfying
\[
fH_M=H_{M'}f.
\]
Such morphisms preserve the joint generalized weight
spaces.

Define the following full subcategories:
\begin{enumerate}
\item
$V$-\textup{wtmod}$_{\geq}$ consists of objects for which
the real parts of all Hamiltonian weights have a
common lower bound:
\[
\inf\left\{
\operatorname{Re}\Delta
\;\middle|\;
M_{\lambda,\Delta}\neq0
\text{ for some }\lambda
\right\}>-\infty.
\]

\item
$V$-\textup{wtmod}$_{\mathcal O}$ consists of objects of
$V$-\textup{wtmod}$_{\geq}$ such that, for every
$\Delta\in\mathbb C$,
\[
\sup\left\{
\operatorname{Re}\lambda
\;\middle|\;
M_{\lambda,\Delta}\neq0
\right\}<\infty.
\]

\item
$V$-\textup{wtmod}$_{KL}$ consists of objects of
$V$-\textup{wtmod}$_{\geq}$ such that every total
generalized Hamiltonian weight space
\[
M_{[\Delta]}=\bigoplus_{\lambda\in\mathbb C}
M_{\lambda,\Delta}
\]
is finite-dimensional.
\end{enumerate}
\end{defn}

Let $\mathfrak g=\mathfrak{sl}_n$ or $\mathfrak{so}_{2n+1}$,
and let $f_{\mathrm{sub}}\in\mathfrak g$ be a subregular
nilpotent element. The universal affine $W$-algebra
$W^{-h^\vee}(\mathfrak g,f_{\mathrm{sub}})$ at the critical
level admits strong generators $J,G^\pm,W_i$, where $J$
has level $-1$ and $G^\pm$ have $J_0$-charges $\pm1$.
The generators $W_i$ have charge zero and Hamiltonian
weight $i$, with
\[
i=2,3,\ldots,n-1 \quad\text{in type }A,
\qquad
i=2,4,\ldots,2n-2 \quad\text{in type }B.
\]
Here the weights refer to the DS Hamiltonian constructed
in Section~2. The $W_i$ may be chosen central.
Together with one additional central field of weight
$n$ in type $A$ and $2n$ in type $B$, they freely generate
the center as a differential algebra.
This additional field can be normalized as a nonzero
multiple of $:G^+G^-:$ plus a normally ordered polynomial
in $J$, the $W_i$, and their derivatives.

By \cite[Theorem~1.1]{Ar1}, the full center $Z$ is the Feigin--Frenkel center and the universal algebra is free over $Z$. Write $Z_+$ for its homogeneous augmentation ideal and set
\[
W_{-h^\vee}(\mathfrak g,f_{\mathrm{sub}})
:=W^{-h^\vee}(\mathfrak g,f_{\mathrm{sub}})/Z_+W^{-h^\vee}(\mathfrak g,f_{\mathrm{sub}}).
\]
This is the simple quotient in \cite[Theorem~1.2]{Ar1}. We use only the augmentation specialization, for which translation and the Hamiltonian descend.

\begin{remark}\label{rk:red}
Both the universal algebra and the centerless quotient have a nonnegative Hamiltonian grading with finite-dimensional homogeneous spaces and semisimple integral $J_0$-action. The positive Heisenberg modes therefore act locally nilpotently, and their oscillator decomposition is a direct sum of ordinary Fock modules; see also \cite[Theorem~1.7.3]{FLM}. Tensoring with $V_{\mathbb Z}$ preserves the Fock decompositions required in Proposition~\ref{prop:relcoh}.
\end{remark}

The quotient map $W^{-h^\vee}(\gg, f_{\mathrm{sub}}) \to W_{-h^\vee}(\gg, f_{\mathrm{sub}})$ induces a canonical homomorphism 
\begin{equation}\label{eq:quotientmap}
(W^{-h^\vee}(\gg, f_{\mathrm{sub}}) \otimes V_{\mathbb Z})^{\tilde J[t]}
\longrightarrow
(W_{-h^\vee}(\gg, f_{\mathrm{sub}}) \otimes V_{\mathbb Z})^{\tilde J[t]}.
\end{equation}
From now on, we will not distinguish between elements and their images under the projection map \eqref{eq:quotientmap}.
Let $\tilde J:=J-:bc:$.  The following result is an immediate corollary of Proposition \ref{prop:relcoh}.

\begin{cor}
Let $W$ be $W^{-h^\vee}(\gg, f_{\mathrm{sub}}) $ or $W_{-h^\vee}(\gg, f_{\mathrm{sub}}) $.
We have a surjective vertex algebra morphism
\begin{equation}\label{eq:surj}
(W \otimes V_{\mathbb Z})^{\tilde J[t]} \twoheadrightarrow H^{\mathrm{rel},0}(W \otimes V_{\mathbb Z}),
\end{equation}
where the kernel is generated by $\tilde J$.
\end{cor}

It is known that $W^{-n}(\mathfrak{sl}_n,f_{\text{sub}})$ is freely generated by the fields $J$, $G^\pm$ and central fields $W_2,\cdots,W_{n-1}$ satisfying the following nontrivial OPEs \cite{GK, LS}
\begin{align}
J(z)J(w)&\sim -(z-w)^{-2},\\ 
G^+(z)G^-(w) &\sim \frac{n! }{(z-w)^{n}} -\frac{n!J(w)}{(z-w)^{n-1}}+\cdots,\\
J(z)G^\pm(w)&\sim \frac{\pm G^\pm}{z-w},
\end{align}
where the remaining terms in second OPE are normally ordered monomials in $J, W_2,\cdots, W_{n-1}$ and their derivatives.
Set in $W^{-n}(\mathfrak{sl}_{n}, f_{\mathrm{sub}}) \otimes V_{\mathbb{Z}}$
\[
X := :G^+ b:,\qquad
Y := :G^- c:.
\]
It is straightforward to verify that both $X$ and $Y$ are annihilated by $\tilde{J}[t]$, and hence belong to the invariant subalgebra
$
\bigl(W^{-n}(\mathfrak{sl}_{n}, f_{\mathrm{sub}}) \otimes V_{\mathbb{Z}}\bigr)^{\tilde{J}[t]}
$
with OPE
\[
X(z)Y(w)
\sim
\frac{n!}{(z-w)^{n+1}}
-
\frac{n!\,\tilde J(w)}{(z-w)^n}
+\cdots.
\]

\begin{lemma}\label{lem:critical-ope}
The odd classes
\[
X=[G^+\otimes b],\qquad Y=[G^-\otimes c]
\]
in $H^{\mathrm{rel},0}(W_{-h^\vee}(\gg, f_{\mathrm{sub}})\otimes V_{\mathbb Z})$ satisfy
\[
X(z)Y(w)\sim\frac{n!}{(z-w)^{n+1}},
\qquad
X(z)X(w)\sim0,\qquad Y(z)Y(w)\sim0.
\]
\end{lemma}

\begin{proof}
Put $a=:bc:$, and denote the Heisenberg vertex subalgebras
generated by $J$ and $a$ by $\pi_0^J$ and $\pi_0^a$,
respectively. In the central augmentation quotient,
all singular coefficients of $G^+(z)G^-(w)$ belong to
$\pi_0^J$. Moreover, the relation obtained by setting
the highest-weight central generator equal to zero
gives $:G^+G^-:\in\pi_0^J$.
This also accounts for the possible simple pole after
multiplication by $b(z)c(w)$.
Thus every singular coefficient of $X(z)Y(w)$ belongs
to $\pi_0^J\otimes\pi_0^a$.

The representatives of $X$ and $Y$ are annihilated by
$\widetilde J[t]$, so their OPE coefficients belong to
$(\pi_0^J\otimes\pi_0^a)^{\widetilde J[t]}$.
On the polynomial realization of this Heisenberg algebra,
\[
\widetilde J_m
=-m\left(
\frac{\partial}{\partial J_{-m}}
+\frac{\partial}{\partial a_{-m}}
\right)
\qquad(m>0).
\]
Consequently,
\[
(\pi_0^J\otimes\pi_0^a)^{\widetilde J[t]}
=
\mathbb C[\widetilde J_{-m}\mid m\geq1]\mathbf1.
\]
Except for the leading scalar $n!$, all singular
coefficients have positive weight and hence lie in
the ideal generated by $\widetilde J$.
They therefore vanish in relative cohomology by
Proposition~\ref{prop:relcoh}.

Finally, the same-sign OPEs of $G^\pm$ are regular
by charge and weight. Together with the regularity
of $b(z)b(w)$ and $c(z)c(w)$, this proves the remaining
assertions.
\end{proof}

Lemma \ref{lem:critical-ope} shows that the fields $X$ and $Y$ generate an odd free field algebra of orthogonal or symplectic type, which is simple. Consequently,
$
H^{\mathrm{rel},0}\!\left(
W_{-n}(\mathfrak {sl}_n,f_{\mathrm{sub}})
\otimes
V_{\mathbb Z}
\right)
$
contains a free field algebra. The following theorem shows that this inclusion is in fact an isomorphism.

\begin{theorem} \label{thm:KScritial}
The Kazama--Suzuki type isomorphisms hold for the large-critical level duality:
\begin{equation} \label{eq:KS1}
H^{\mathrm{rel},p}\!\left(
  W_{-n}(\mathfrak{sl}_n, f_{\mathrm{sub}})\otimes V_{\mathbb Z}
\right)
\cong  \delta_{p,0} W_{\infty}(\mathfrak{sl}_{n|1}).
\end{equation}
\end{theorem}

\begin{proof}
 Lemma \ref{lem:critical-ope} gives an embedding of $F_n$ into the degree-zero reduction. We show that it is onto by an ordinary-character calculation.
Put
\[
\ch V(q,z)=\operatorname{tr}_V(q^Hz^{J_0}).
\]
The PBW character of the universal critical subregular algebra is
\[
\ch W^{-n}(\mathfrak{sl}_n,f_{\mathrm{sub}})
=\frac{1}{\prod_{i=1}^{n-1}(q^i;q)_\infty
(zq^{n/2};q)_\infty(z^{-1}q^{n/2};q)_\infty}.
\]
By \cite[Theorem~1.1]{Ar1} this algebra is free over its center, whose generators have weights $2,\ldots,n$. Its augmentation quotient therefore has character
\[
\ch B_n(q,z)=
\frac{(q^n;q)_\infty}
{(q;q)_\infty(zq^{n/2};q)_\infty(z^{-1}q^{n/2};q)_\infty}.
\]
Here the character division uses freeness over the center, not merely the identification of the center. The standard lattice fermions have weights $1/2,1/2$ and opposite charges. The relative ghosts have no zero-mode factor and have Euler character $(q;q)_\infty^2$. Thus the relative complex has Euler character
\[
\frac{(q^n;q)_\infty(q;q)_\infty^2}{(q;q)_\infty}
\left.
\frac{(-zq^{1/2};q)_\infty(-z^{-1}q^{1/2};q)_\infty}
{(zq^{n/2};q)_\infty(z^{-1}q^{n/2};q)_\infty}
\right|_{z^0}.
\]
The total Hamiltonian is $H+L^{\mathrm{lat}}_0+L^{\mathrm{gh}}_0$; the BRST current has weight one, so the differential preserves it. Every homogeneous space is finite-dimensional and the Euler sum is well defined degree by degree. Lemma~\ref{lem:vanish} shows that cohomology is concentrated in degree zero. By~\eqref{eq:comb1} at $y=1$, the displayed expression equals
\[
(-q^{(n+1)/2};q)_\infty^2=\ch F_n(q).
\]
It compares ordinary graded dimensions. The injective map from $F_n$ is therefore surjective in every degree, proving~\eqref{eq:KS1} and the higher-cohomology vanishing.
\end{proof}

\noindent Together with Remark \ref{rk:red} and  Theorem \ref{thm:isoalgebra2} with $V=W_{-n}(\mathfrak{sl}_{n}, f_{\mathrm{sub}})$, we obtain the following vertex algebra isomorphism. 

\begin{cor}
We have the vertex algebra isomorphism
\begin{equation}\label{eq:iso2}
(W_{\infty}(\mathfrak{sl}_{n|1})\otimes V_{\sqrt{-1}\mathbb Z} )^{\mathbb C^*}\cong W_{-n}(\mathfrak{sl}_{n},f_{\mathrm{sub}}). 
\end{equation}
\end{cor}

\begin{cor}\label{thm:invgen} 
The invariant vertex algebra 
$
(W_{-n}(\mathfrak {sl}_{n}, f_{\mathrm{sub}}) \otimes V_{\mathbb Z})^{\tilde J[t]}
$
is strongly generated by $\tilde J$, $X$ and $Y$.
\end{cor}

\begin{proof}
Set
\[
V=W_{-n}(\mathfrak{sl}_n,f_{\mathrm{sub}})
   \otimes V_{\mathbb Z},
\qquad
C=V^{\widetilde J[t]}.
\]
By Proposition~\ref{prop:relcoh}, the natural graded
homomorphism
\[
\rho:C\longrightarrow H^{\mathrm{rel},0}(V)
\]
is surjective, with
\[
\ker\rho=(\widetilde J)
=\sum_{r\geq1}\widetilde J_{(-r)}C.
\]
Here $\widetilde J$ is central in $C$, so its generated
ideal consists precisely of these finite sums.

Let $A$ be the linear span of normally ordered monomials
in $\widetilde J,X,Y$ and their derivatives. Since these
fields belong to $C$, we have $A\subset C$. We prove the
reverse inclusion by induction on the nonnegative
half-integral Hamiltonian weight.

Let $v\in C$ be homogeneous of weight $d$.
By Theorem~\ref{thm:KScritial},
$H^{\mathrm{rel},0}(V)$ is strongly generated by
$\rho(X)$ and $\rho(Y)$. Hence there is a homogeneous
normally ordered polynomial $P(X,Y)$, allowing
derivatives, such that
\[
\rho(v)=\rho(P(X,Y)).
\]
The kernel description gives
\[
v-P(X,Y)
=\sum_{r\geq1}\widetilde J_{(-r)}v_r,
\qquad v_r\in C_{d-r},
\]
with only finitely many nonzero terms.
Since $\operatorname{wt}(\widetilde J)=1$, every $v_r$
has strictly smaller weight. By induction, $v_r\in A$,
and therefore
\[
\widetilde J_{(-r)}v_r
=\frac{1}{(r-1)!}
:\partial^{r-1}\widetilde J\,v_r:
\in A.
\]
For $d=0$, all $C_{d-r}$ vanish, providing the base
case. Thus $v\in A$, and $C=A$, proving the asserted
strong generation.
\end{proof}

Since  $Z=Z(W^{-n}(\mathfrak{sl}_{n},f_{\mathrm{sub}}))$ is freely generated by the central elements $W_i$ for $2\leq i \leq n$,
we have the following result.

\begin{cor}\label{cor:universalcommutant}
The commutant $(W^{-n}(\mathfrak{sl}_n,f_{\mathrm{sub}})\otimes V_{\mathbb Z})^{\widetilde J[t]}$ is strongly generated by $\widetilde J,X,Y$ and the central generators $W_2,\ldots,W_n$.
\end{cor}

\begin{proof}
Set
\[
V=W^{-n}(\mathfrak{sl}_n,f_{\mathrm{sub}})\otimes V_{\mathbb Z},
\qquad
C=V^{\widetilde J[t]}.
\]
Let $Z_+$ be the augmentation ideal of the center
$
Z=\mathbb C[\partial^m W_i
  \mid 2\leq i\leq n,\ m\geq0].
$
Equip $V$ and $C$ with the decreasing filtrations
\[
\mathcal F^rV=(Z_+)^rV,
\qquad
\mathcal F^rC=C\cap\mathcal F^rV.
\]
Since $Z$ is central, this filtration is compatible
with translation and all vertex products. Moreover,
freeness of $W$ over $Z$ gives an isomorphism of
graded vertex algebras
\[
\operatorname{gr}_{\mathcal F}V
\cong
\overline V\otimes\operatorname{gr}_{Z_+}Z,
\qquad
\overline V
=
W_{-n}(\mathfrak{sl}_n,f_{\mathrm{sub}})
\otimes V_{\mathbb Z}.
\]
Here $\operatorname{gr}_{Z_+}Z$ is the polynomial
differential algebra generated by the symbols of
$W_2,\ldots,W_n$, each of filtration degree one.

We next justify passing to $\widetilde J[t]$-invariants.
The ordinary-Fock decomposition of $V$ makes it
semisimple as a module for the commuting nondegenerate
Heisenberg fields $J$ and $:bc:$.
The spaces $\mathcal F^rV$ are stable under both fields,
so the sequences
\[
0\longrightarrow\mathcal F^{r+1}V
\longrightarrow\mathcal F^rV
\longrightarrow
\mathcal F^rV/\mathcal F^{r+1}V
\longrightarrow0
\]
split as modules for their Heisenberg tensor product.
Taking $\widetilde J[t]$-invariants is therefore exact
on these sequences. Since the Heisenberg fields act
trivially on $Z$, we obtain
\[
\operatorname{gr}_{\mathcal F}C
\cong
\overline V^{\widetilde J[t]}
\otimes\operatorname{gr}_{Z_+}Z.
\]
By Corollary~\ref{thm:invgen}, this associated graded
vertex algebra is strongly generated by the symbols
of
\[
\widetilde J,\quad X,\quad Y,\quad W_2,\ldots,W_n.
\]

Finally, lifting from the associated graded terminates
in each Hamiltonian weight. Indeed, every generator
of $Z_+$ has weight at least two, while $V$ has
nonnegative weights. Hence
\[
(\mathcal F^rV)_d=0
\qquad\text{whenever }2r>d.
\]
For a homogeneous $v\in C_d$, express its class modulo
$\mathcal F^1C$ as a normally ordered polynomial in
the indicated symbols and subtract the corresponding
polynomial in their lifts. The remainder belongs to
$(\mathcal F^1C)_d$. Repeating this procedure in
successive filtration degrees gives zero after
finitely many steps. Thus $v$ is a normally ordered
polynomial in $\widetilde J,X,Y,W_2,\ldots,W_n$
and their derivatives, proving the assertion.
\end{proof}

For $\mathfrak{g}=\mathfrak{so}_{2n+1}$, the OPEs of the strong generators remain unknown, in particular those involving $G^+$ and $G^-$.

\begin{theorem} \label{thm:OPE}
The subregular $W$-algebra $W^{-2n+1}(\mathfrak{so}_{2n+1}, f_{\mathrm{sub}})$ satisfies the following OPE relations:
\begin{align}\label{eq:JJ}
J(z)J(w) &\sim -(z-w)^{-2},\\ \label{eq:GG}
G^+(z)G^-(w) &\sim \frac{c}{(z-w)^{2n}} - \frac{c \, J(w)}{(z-w)^{2n-1}} + \cdots,\\
J(z)G^\pm(w) &\sim \frac{\pm G^\pm(w)}{z-w}, \label{eq:JG}
\end{align}
where $c$ is a nonzero constant and the omitted terms in the second OPE are normally ordered polynomials in $J$, $W_2,\dots,W_{2n-2}$, and their derivatives.
\end{theorem}

\begin{proof}
By choosing a suitable normalization of $ J$, it suffices to show the relation (\ref{eq:GG}). 
According to the conformal weight computation,   the OPE for $G^+(z)$ and $G^-(w)$ has the form 
\[G^+(z)G^-(w) \sim \frac{c}{(z-w)^{2n}} + \frac{c' \, J(w)}{(z-w)^{2n-1}} +\cdots\]
 for certain constants $c,c'\in \mathbb C$. 
One deduces that $c=-c'$, thanks to the relation that
\[
-c'=J_{(1)}G^+_{(2n-2)}G^- =[J_{(1)},G^+_{(2n-2)}]G^-=G^+_{(2n-1)}G^-=c.
\]
Assume that $c=0$.
We consider the corresponding centerless quotient $W_{-2n+1}(\mathfrak{so}_{2n+1}, f_{\mathrm{sub}})$, which is simple by \cite{Ar1}. The OPE for (\ref{eq:JJ}) and (\ref{eq:JG}) remains the same and the OPE for (\ref{eq:GG}) turns into
\[G^+(z)G^-(w) \sim \frac{c}{(z-w)^{2n}} - \frac{c\, J(w)}{(z-w)^{2n-1}} +\sum_{i=1}^{2n-2} \frac{f_i (w)} {(z-w)^{2n-1-i}},\]
where $f_i$ is a polynomial  on the variables $J$ and its derivatives with each monomials of conformal weight  $i+1$. 
Set $f_1(w)=a_1 J^2(w)+a_2 \partial J(w)$. Then by a simple computation
\[
-2a_2=J_{(2)} G^+_{(2n-3)}G^-=[J_{(2)},G^+_{(2n-3)}]G^-=G^+_{(2n-1)}G^-=0,
\]
and 
\[
2a_1=J_{(1)}^2 G^+_{(2n-3)}G^-=J_{(1)}[J_{(1)},G^+_{(2n-3)}]G^-=J_{(1)}G^+_{(2n-2)}G^-=0,
\]
which implies $G^+_{(2n-3)}G^-=0$. By applying the the positive modes of $J(z)$ recursively, one can show that all of $f_i=0$ and hence the OPE for $G^+(z)$ and $G^-(w)$ in $W_{-2n+1}(\mathfrak{so}_{2n+1}, f_{\mathrm{sub}})$ is regular. 
Hence $G^+$ and $G^-$ generate a nontrivial proper ideal  in the centerless quotient, which contradicts to the simplicity of $W_{-2n+1}(\mathfrak{so}_{2n+1}, f_{\mathrm{sub}})$. Hence $c\neq 0$.
\end{proof}

Using Theorem \ref{thm:OPE}, one obtains an analogous Kazama-Suzuki type isomorphism in the large and critical level duality setting, by mimicking the proof of Theorem~\ref{thm:KScritial}.
\begin{theorem}  \label{thm:KSneq}
The Kazama-Suzuki type isomorphisms for the pair $(\mathfrak g_1,\mathfrak g_2)=
 (\mathfrak{so}_{2n+1},\mathfrak{osp}_{2|2n})$ hold in the large-critical level duality, namely
\begin{align*}
H^{\mathrm{rel},p}\!\left(
  W_{-2n+1}(\mathfrak{so}_{2n+1}, f_{\mathrm{sub}})\otimes V_{\mathbb Z}
\right)
&\cong  \delta_{p,0} W_{\infty}(\mathfrak{osp}_{2|2n}),\\
(W_{\infty}(\mathfrak{osp}_{2|2n})\otimes V_{\sqrt{-1}\mathbb Z} )^{\mathbb C^*} &\cong W_{-2n+1}(\mathfrak{so}_{2n+1},f_{\mathrm{sub}}). 
\end{align*}
\end{theorem}
\begin{proof}
Normalize one charged generator so that the nonzero scalar in Theorem~\ref{thm:OPE} is $1$. The same Heisenberg cancellation as in type $A$ shows that the BRST classes of $:G^+b:$ and $:G^-c:$ have only the scalar pole of order $2n+1$. They therefore generate an embedded copy of $F_{2n}$.

The universal algebra has even PBW generators of weights $1,2,4,\ldots,2n-2$ and two charged generators of weight $n$. Dividing its character by the character of the freely acting center, with weights $2,4,\ldots,2n$, gives
\[
\ch W_{-2n+1}(\mathfrak{so}_{2n+1},f_{\mathrm{sub}})(q,z)
=\frac{(q^{2n};q)_\infty}
{(q;q)_\infty(zq^n;q)_\infty(z^{-1}q^n;q)_\infty}.
\]
The relative Euler character is obtained by multiplying by the lattice-fermion factors and $(q;q)_\infty^2$ and taking the charge-zero coefficient. By~\eqref{eq:comb2} at $y=1$, it is $(-q^{n+1/2};q)_\infty^2=\ch F_{2n}$. Cohomology is concentrated in degree zero by Lemma~\ref{lem:vanish}, so the embedding is surjective. Proposition~\ref{prop:OPE} identifies this free field with $W_\infty(\mathfrak{osp}_{2|2n})$. The inverse orbifold statement follows from Theorem~\ref{thm:isoalgebra2} and the Fock decomposition in Remark~\ref{rk:red}.
\end{proof}

By the proofs of Theorem \ref{thm:KScritial} and \ref{thm:KSneq}, both the $W_{\infty}(\mathfrak{sl}_{2n|1})$ and $W_{\infty}(\mathfrak{osp}_{2|2n})$ are isomorphic to the free field algebra generated by two odd fields of weight $n+\frac{1}{2}$. Hence the corresponding centerless critical level subregular $W$-algebras are also isomorphic.

\begin{cor}
There is an isomorphism of vertex algebras
\[
W_{-2n}(\mathfrak{sl}_{2n},f_{\mathrm{sub}})
\cong
W_{-2n+1}(\mathfrak{so}_{2n+1},f_{\mathrm{sub}}).
\]
\end{cor}

In the large-critical level setting, the Kazama-Suzuki type isomorphism for universal affine $W$-algebras does not hold in general; nevertheless, one still obtains a canonical embedding map between the corresponding vertex algebras. We need the following lemma.

\begin{lemma}\label{lem:large-limit-generators}
The charged generators of $W^\infty(\mathfrak{sl}_{n|1})$
can be normalized so that
\[
G^{+,\infty}(z)G^{-,\infty}(w)
\sim
-\frac{1}{(z-w)^{n+1}}
+\sum_{i=1}^{n}\frac{A_i(w)}{(z-w)^{n+1-i}}.
\]
The fields $A_1,\ldots,A_n$ are central, and
\[
A_1,\ldots,A_n,\quad G^{+,\infty},G^{-,\infty}
\]
freely strongly generate
$W^\infty(\mathfrak{sl}_{n|1})$.
\end{lemma}

\begin{proof}
Choose neutral PBW generators $w_1,\ldots,w_n$
corresponding to
\[
h_0=\sum_{a=1}^{n}E_{a,a}+nE_{n+1,n+1},
\qquad f,f^2,\ldots,f^{n-1},
\qquad
f=\sum_{a=1}^{n-1}E_{a+1,a}.
\]
The large-level construction makes these fields central.
By weight considerations,
\[
A_i=c_iw_i+P_i(w_1,\ldots,w_{i-1}),
\]
where $P_i$ is a differential polynomial. It remains
to verify that $c_i\neq0$.

The regular family identifies the limiting charged bracket
with the classical Drinfeld--Sokolov bracket after setting
the charged variables equal to zero by
Proposition \ref{prop:OPE}.
We apply \cite[Theorem~2.4]{PS} to the charged generators
corresponding to $E_{n,n+1}$ and $E_{n+1,1}$.
The two families of odd matrices
\[
E_{r+1,n+1},\qquad E_{n+1,r+1},
\qquad 0\leq r\leq n-1,
\]
are dual with respect to the supertrace form.
The scalar contractions give a nonzero coefficient of $\xi^n$.
In the coefficient of $\xi^{n-i}$, the terms linear in
the neutral generators and containing no derivatives
arise by replacing one scalar contraction with a matrix
bracket. These terms have a common nonzero scalar factor,
and their bracket contributions sum to
\[
\sum_{r=0}^{n-i}
[E_{n+1,r+1},E_{r+i,n+1}]
=
\begin{cases}
h_0,&i=1,\\
f^{i-1},&2\leq i\leq n.
\end{cases}
\]
Hence $c_i\neq0$ for every $i$.
The change from $w_1,\ldots,w_n$ to $A_1,\ldots,A_n$
is therefore an invertible triangular change of
differential generators, and the assertion follows
from the PBW theorem.
\end{proof}

\begin{theorem} \label{thm:emb}
There exists an embedding of the universal affine W-algebra associated to $\mathfrak{sl}_n$ at the critical level into the orbifold vertex algebra of $W^{\infty}(\mathfrak{sl}_{n|1})$, namely
\begin{equation} \label{eq:emb}
W^{-n}(\mathfrak{sl}_n,f_{\mathrm{sub}}) \hookrightarrow  (W^{\infty}(\mathfrak{sl}_{n|1})\otimes V_{\sqrt{-1}\mathbb Z} )^{\mathbb C^*} 
\end{equation}
\end{theorem}

\begin{proof}
Let $A_0=-\mathbf1$. Define central fields $Q_i$ by
\[
\sum_{i=0}^{n}Q_it^i
\equiv
E_{A_1}(t)\sum_{i=0}^{n}A_it^i
\pmod{t^{n+1}},
\]
where, for a even field $a$,
$
E_a(t):=
\exp\left(\sum_{r\geq1}\frac{a_{(-r)}}r t^r\right).
$
The first three coefficients are
\[
Q_0=-\mathbf1,\qquad Q_1=0,\qquad
Q_2=A_2+\frac12A_1^2-\frac12\partial A_1.
\] 
More generally, $Q_i-A_i$ is a differential polynomial
in $A_1,\ldots,A_{i-1}$ for $2\leq  i \leq n$.
It follows from
Lemma~\ref{lem:large-limit-generators} that
$A_1,Q_2,\ldots,Q_n$ freely generate the central
differential algebra.

Let $(\lambda,\lambda)=-1$, and consider
$B=
\left(
W^\infty(\mathfrak{sl}_{n|1})
\otimes V_{\mathbb Z\lambda}
\right)^{\mathbb C^*}.
$
The fields
\[
H=A_1-\lambda,\qquad
X=G^{+,\infty}\otimes e^\lambda,\qquad
Y=-\epsilon(\lambda,-\lambda)^{-1}
G^{-,\infty}\otimes e^{-\lambda}
\]
belong to $B$. The lattice vertex-operator formula,
with the tensor-product super sign, gives
\begin{equation}\label{eq:universal-lattice-product}
Y_B(X,t)Y
=
t^{-n}E_{-H}(t)
\left(-\mathbf1+\sum_{i=2}^{n}Q_it^i\right)
+O(t).
\end{equation}
Here we have used
\[
E_\lambda(t)\sum_{i=0}^{n}A_it^i
\equiv
E_{\lambda-A_1}(t)\sum_{i=0}^{n}Q_it^i
\pmod{t^{n+1}}.
\]
Since $A_1$ is central, the remaining OPEs are
\[
H(z)H(w)\sim-\frac1{(z-w)^2},\qquad
H(z)X(w)\sim\frac{X(w)}{z-w},\qquad
H(z)Y(w)\sim-\frac{Y(w)}{z-w}.
\]
The same-sign OPEs of $X$ and $Y$ are regular.

The OPE comparison in
\cite[Lemma~4.1 and Theorem~4.2]{LS}
now gives a homomorphism
\[
\varphi:
W^{-n}(\mathfrak{sl}_n,f_{\mathrm{sub}})
\longrightarrow B,
\qquad
J\longmapsto H,\quad G^+\longmapsto X,\quad
G^-\longmapsto Y,
\]
where the source charged generators are normalized
to have leading mixed-OPE coefficient $-1$.
The central generators $W_2,\ldots,W_n$ of the source
are normalized by
\[
t^nE_J(t)Y(G^+,t)G^-
\equiv
-\mathbf1+\sum_{i=2}^{n}W_it^i
\pmod{t^{n+1}}.
\]
Here $W_n$ denotes the additional central generator
of weight $n$ described above.
Since $\varphi(J)=H$ and
$\varphi(G^\pm)=X,Y$, respectively, applying
$\varphi$ to this identity gives
\[
t^nE_H(t)Y_B(X,t)Y
\equiv
-\mathbf1+\sum_{i=2}^{n}\varphi(W_i)t^i
\pmod{t^{n+1}}.
\]
Comparison with
\eqref{eq:universal-lattice-product} therefore yields
$
\varphi(W_i)=Q_i$ for $2\leq i\leq n$.

It remains to prove injectivity. Set
\[
R=\mathbb C[\partial^mQ_i\mid
2\leq i\leq n,\ m\geq0].
\]
The PBW basis and the neutrality of the $Q_i$
show that $B$ is free over $R$.
Thus $B/R_+B$ is nonzero, and the induced map
\[
\overline\varphi:
W_{-n}(\mathfrak{sl}_n,f_{\mathrm{sub}})
\longrightarrow B/R_+B
\]
is injective because its source is simple.
The map $\varphi$ identifies the source center
with $R$. Since both sides are free over this
algebra, the associated graded map for the
central augmentation filtrations is
$\overline\varphi\otimes\operatorname{id}_{\operatorname{gr}R}$,
hence injective.
The source filtration is separated because its
central generators have positive weights.
Therefore $\varphi$ is injective.
\end{proof}

\begin{conj}
Similar embedding (\ref{eq:emb2}) exists for the pair $(\mathfrak {so}_{2n+1},\mathfrak{osp}_{2|2n})$, namely 
\begin{equation} \label{eq:emb2}
W^{-2n+1}(\mathfrak {so}_{2n+1},f_{\mathrm{sub}}) \hookrightarrow  (W^{\infty}(\mathfrak{osp}_{2|2n})\otimes V_{\sqrt{-1}\mathbb Z} )^{\mathbb C^*} 
\end{equation}

\end{conj}

\subsection{Representation-theoretic correspondence}
\label{sec:repncorr}

We describe the modules of the free-field algebra $F_N$
and apply the equivalence of Section~\ref{sec:equiv}
to the centerless critical-level $W$-algebras.
For $N\geq2$, write
\[
B_N:=W_{-N}(\mathfrak{sl}_N,f_{\mathrm{sub}}).
\]
By Theorem~\ref{thm:KScritical}, the same description
applies to
$W_{-2n+1}(\mathfrak{so}_{2n+1},f_{\mathrm{sub}})$
with $N=2n$.

\begin{defn}\label{def:BN-categories}
Let $\pi^J$ be the Heisenberg vertex subalgebra of $B_N$
generated by $J$, and let $\pi^J_m$ denote its Fock
module of $J_0$-charge $m$.
Define $\Bcat_N$ to be the full subcategory
of weak $B_N$-modules consisting of finite-length
modules $M$ whose restriction to $\pi^J$ has the form
\[
M\big|_{\pi^J}
\cong
\bigoplus_{m\in\mathbb Z}\pi^J_m\otimes E_m
\]
for some multiplicity spaces $E_m$.
Morphisms are $B_N$-module homomorphisms.
\end{defn}

A chosen compatible Hamiltonian is not part of the
data of an object of $\Bcat_N$.
Our block, projective-cover, and extension statements
will be made in this category.
At the end of the section, we relate these results
to $B_N$-\textup{wtmod}, whose objects carry chosen
compatible Hamiltonians and whose morphisms intertwine
them, as in Definition~\ref{defcatwtmod}.

Fix $N\geq1$ and put
\[
d=\frac{N+1}{2},
\qquad
a=\frac{N-1}{2}.
\]
The odd generators of $F_N$ have expansions and OPE
\[
X^\pm(z)
=
\sum_{s\in d+\mathbb Z}X^\pm_s z^{-s-d},
\qquad
X^+(z)X^-(w)\sim(z-w)^{-N-1}.
\]
Let $\mathfrak g_N$ be the associated mode Lie
superalgebra. After fixing its even central element
to act as $1$, its relations are
\begin{equation}\label{eq:comm}
\{X^+_s,X^-_t\}
=
\binom{s+a}{N}\delta_{s+t,0},
\qquad
\{X^+_s,X^+_t\}
=
\{X^-_s,X^-_t\}
=
0.
\end{equation}
The $2N$ modes
\[
X^\pm_s,
\qquad
s=-a,-a+1,\ldots,a,
\]
are odd central elements. We retain the notation
\[
Z(\mathfrak g_N)
:=
\Lambda(X^\pm_s\mid s=-a,-a+1,\ldots,a),
\qquad
\deg X^\pm_s=\pm1.
\]
Thus $Z(\mathfrak g_N)$ denotes the supercenter of
the associative mode algebra obtained from
$U(\mathfrak g_N)$ by specializing its scalar central
element to $1$. In particular,
\[
Z(\mathfrak g_1)=\Lambda(X^+_0,X^-_0),
\qquad
Z(\mathfrak g_2)
=
\Lambda
\bigl(
X^+_{1/2},X^+_{-1/2},
X^-_{-1/2},X^-_{1/2}
\bigr).
\]
Set
\[
\mathfrak g_{N,+}
=
\operatorname{span}\{X^\pm_s:s>a\},
\qquad
\mathfrak g_{N,-}
=
\operatorname{span}\{X^\pm_s:s<-a\}.
\]
The pairing between these spaces in~\eqref{eq:comm}
is nondegenerate. Choose paired odd bases, consisting
of scalar multiples of the modes,
$p_i\in\mathfrak g_{N,+}$ and
$q_i\in\mathfrak g_{N,-}$ with
\[
\{p_i,q_j\}=\delta_{ij}.
\]
For a weak $F_N$-module $M$, define
\[
\Omega(M)
=
\{m\in M:\mathfrak g_{N,+}m=0\}.
\]
This is a $Z(\mathfrak g_N)$-supermodule.
Conversely, for a $Z(\mathfrak g_N)$-supermodule $U$,
let
\[
\mathcal I_N(U)
=
\Lambda(\mathfrak g_{N,-})\otimes U,
\]
where the $q_i$ act by exterior multiplication,
the $p_i$ by contraction, and the central modes
act on $U$ with the usual tensor-product signs.
The resulting restricted mode representation is
a weak $F_N$-module.

\begin{theorem}\label{thm:irred}
For every weak $F_N$-module $M$, the action map
\[
\mathcal I_N(\Omega(M))
\longrightarrow M,
\qquad
q_{i_1}\cdots q_{i_k}\otimes u
\longmapsto
q_{i_1}\cdots q_{i_k}u,
\]
is an isomorphism.
The functors $\Omega$ and $\mathcal I_N$ are mutually
inverse and preserve charge gradings.
Moreover, $M$ is finitely generated if and only if
$\Omega(M)$ is finite-dimensional.
\end{theorem}

\begin{proof}
Put $P_i=1-q_ip_i$. These are commuting even projections
and $p_iP_i=0$.
For $m\in M$, restrictedness gives a finite set $I$
such that every $p_i$ with $i\notin I$ annihilates $m$.
Expanding
\[
1=\prod_{i\in I}(P_i+q_ip_i)
\]
expresses $m$ as a finite sum of vectors $q_Ju_J$
with $u_J\in\Omega(M)$.
Modes outside $I$ still annihilate these coefficients,
since they anticommute with the modes indexed by $I$.
This proves surjectivity.

For a finite relation $\sum_Jq_Ju_J=0$, choose $I$
containing all its oscillator indices.
Suitable products of the $p_i$, followed by
$\prod_{i\in I}P_i$, extract each coefficient $u_J$,
up to sign. Hence the action map is injective.
The construction is natural and gives
\[
\Omega(\mathcal I_N(U))=U.
\]

If $M$ is finitely generated, expand a finite set
of generators in this decomposition.
Their finitely many vacuum coefficients generate
$\Omega(M)$ over the finite-dimensional algebra
$Z(\mathfrak g_N)$, so $\Omega(M)$ is finite-dimensional.
Conversely, a basis of a finite-dimensional $U$
generates $\mathcal I_N(U)$.
All maps preserve charge because the modes are
charge homogeneous.
\end{proof}

\begin{cor}\label{thm:classification}\label{lem:mod}
The category of finitely generated charge-graded
$F_N$-supermodules is equivalent to the category of
finite-dimensional charge-graded
$Z(\mathfrak g_N)$-supermodules, via $\Omega$ and
$\mathcal I_N$.
Its simple objects are the vacuum module with
charge and parity shifts.
\end{cor}

\begin{proof}
The equivalence follows from Theorem~\ref{thm:irred}.
The augmentation ideal of $Z(\mathfrak g_N)$ is
nilpotent and its quotient is $\mathbb C$.
Thus every simple graded supermodule is a
one-dimensional trivial module concentrated in a
single charge and parity.
Its induced module is the vacuum with the
corresponding shifts.
\end{proof}

For comparison with $B_N$-modules, we restrict to
the parity convention
$
p(u)=\deg u\pmod2.
$
Write $U\langle r\rangle$ for the shift increasing
charge by $r$ and parity by $r\pmod2$.

\begin{remark}\label{rm:loewy}
The Loewy structure of a finitely generated
$F_N$-module is completely determined by the
Loewy structure of its vacuum space as a
$Z(\mathfrak g_N)$-module.
More precisely, the exact equivalence above
identifies their radical and socle filtrations.

The regular module $Z(\mathfrak g_N)$ is projective.
Since $Z(\mathfrak g_N)$ is local, its augmentation
map onto the trivial module is a projective cover.
If $\mathfrak m$ is its augmentation ideal, then
\[
\operatorname{rad}^jZ(\mathfrak g_N)=\mathfrak m^j,
\qquad
\mathfrak m^j/\mathfrak m^{j+1}
\cong
\Lambda^j
\bigl(
\mathbb C^N_{+1}\oplus\mathbb C^N_{-1}
\bigr).
\]
In particular, its Loewy length is $2N+1$.

The monomial diagram for $Z(\mathfrak g_2)$ is
displayed below.
Each node denotes a monomial class in a radical
layer; the coloured arrows record multiplication
by the four generators, with exterior signs
suppressed.
Its rows have dimensions $1,4,6,4,1$.
\end{remark}

\begin{center}
\resizebox{0.98\linewidth}{!}{%
\begin{tikzpicture}[->, thick, xscale=0.95, yscale=0.72]
    \node (1) at (0,6) {$1$};
    \node (21) at (-4.5,3) {$X^+_{\frac12}$};
    \node (22) at (-1.5,3) {$X^+_{-\frac12}$};
    \node (23) at (1.5,3) {$X^-_{-\frac12}$};
    \node (24) at (4.5,3) {$X^-_{\frac12}$};

    \node (31) at (-7.5,0)
    {$X^+_{\frac12}X^+_{-\frac12}$};
    \node (32) at (-4.5,0)
    {$X^+_{\frac12}X^-_{-\frac12}$};
    \node (33) at (-1.5,0)
    {$X^+_{\frac12}X^-_{\frac12}$};
    \node (34) at (1.5,0)
    {$X^+_{-\frac12}X^-_{-\frac12}$};
    \node (35) at (4.5,0)
    {$X^+_{-\frac12}X^-_{\frac12}$};
    \node (36) at (7.5,0)
    {$X^-_{-\frac12}X^-_{\frac12}$};

    \node (41) at (-4.5,-3)
    {$X^+_{\frac12}X^+_{-\frac12}X^-_{-\frac12}$};
    \node (42) at (-1.5,-3)
    {$X^+_{\frac12}X^+_{-\frac12}X^-_{\frac12}$};
    \node (43) at (1.5,-3)
    {$X^+_{\frac12}X^-_{-\frac12}X^-_{\frac12}$};
    \node (44) at (4.5,-3)
    {$X^+_{-\frac12}X^-_{-\frac12}X^-_{\frac12}$};
    \node (5) at (0,-6)
    {$X^+_{\frac12}X^+_{-\frac12}
      X^-_{-\frac12}X^-_{\frac12}$};

    \draw[red] (1) -- (21);
    \draw[green] (1) -- (22);
    \draw[blue] (1) -- (23);
    \draw[mauve] (1) -- (24);

    \draw[green] (21) -- (31);
    \draw[blue] (21) -- (32);
    \draw[mauve] (21) -- (33);
    \draw[red] (22) -- (31);
    \draw[blue] (22) -- (34);
    \draw[mauve] (22) -- (35);
    \draw[red] (23) -- (32);
    \draw[green] (23) -- (34);
    \draw[mauve] (23) -- (36);
    \draw[red] (24) -- (33);
    \draw[green] (24) -- (35);
    \draw[blue] (24) -- (36);

    \draw[blue] (31) -- (41);
    \draw[mauve] (31) -- (42);
    \draw[green] (32) -- (41);
    \draw[mauve] (32) -- (43);
    \draw[green] (33) -- (42);
    \draw[blue] (33) -- (43);
    \draw[red] (34) -- (41);
    \draw[mauve] (34) -- (44);
    \draw[red] (35) -- (42);
    \draw[blue] (35) -- (44);
    \draw[red] (36) -- (43);
    \draw[green] (36) -- (44);

    \draw[mauve] (41) -- (5);
    \draw[blue] (42) -- (5);
    \draw[green] (43) -- (5);
    \draw[red] (44) -- (5);

    \node[align=center] at (0,-7.25)
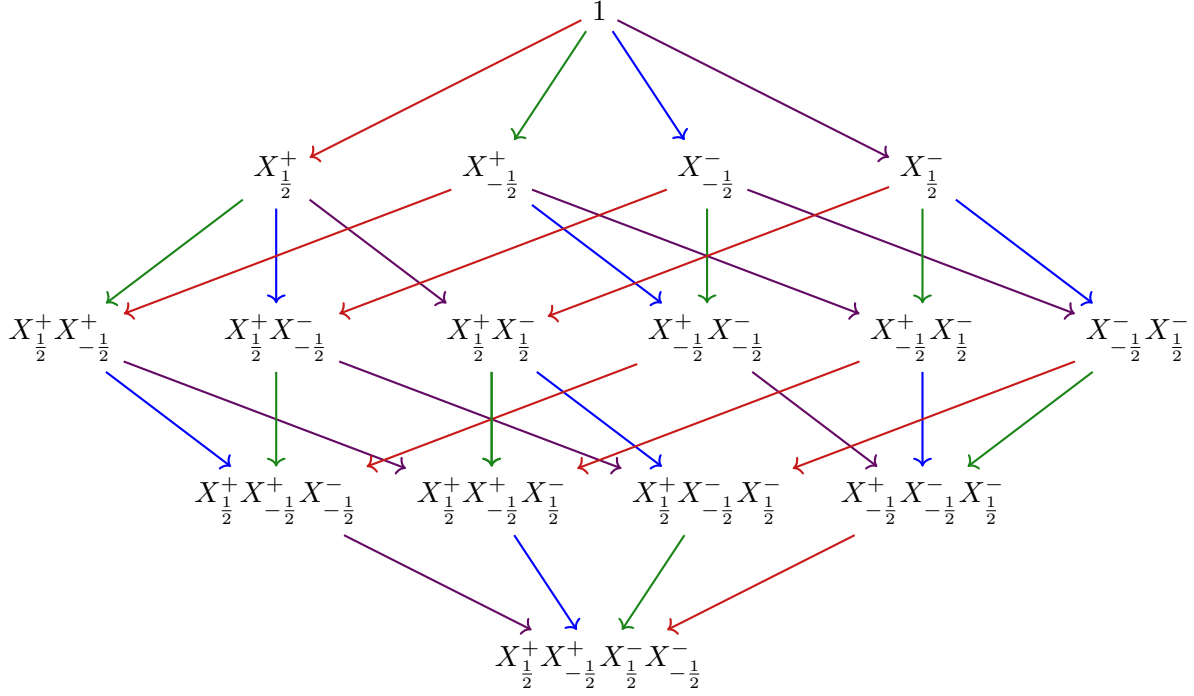
    {Loewy diagram for $Z(\mathfrak g_2)$};
\end{tikzpicture}%
}
\end{center}

\begin{theorem}\label{thm:block-main}\label{thm:projectives}
For $N\geq2$, the functor $F_0\mathcal I_N$ gives
an exact equivalence
\[
Z(\mathfrak g_N)\textup{-gmod}_{\mathrm{fd}}
\xrightarrow{\;\sim\;}
\Bcat_N,
\]
where the category on the left uses the parity
convention above.
The category $\Bcat_N$ is also the
finitely generated part of the integral-charge
Fock category, and it is an indecomposable block.

Its simple objects and their projective covers are
\[
S_r=F_0\bigl(F_N\langle r\rangle\bigr),
\qquad
P_r=
F_0\mathcal I_N
\bigl(Z(\mathfrak g_N)\langle r\rangle\bigr),
\qquad r\in\mathbb Z.
\]
Here $S_0=B_N$, and the $S_r$ are the
spectral-flow images of the vacuum.
The radical layers of $P_r$ correspond to
\[
(\mathfrak m^j/\mathfrak m^{j+1})\langle r\rangle
\cong
\Lambda^j
\bigl(
\mathbb C^N_{+1}\oplus\mathbb C^N_{-1}
\bigr)\langle r\rangle.
\]
In particular, every $P_r$ has Loewy length $2N+1$.
\end{theorem}

\begin{proof}
The equivalence of Section~\ref{sec:equiv}, applied to
$
B_N\cong
(F_N\otimes V_{i\mathbb Z})^{\mathbb C^*},
$
identifies the integral-charge Fock category with
charge-graded $F_N$-modules in the specified parity
sector.
Corollary~\ref{thm:classification} identifies its
finite-length subcategory with the displayed
finite-dimensional exterior-algebra category.

Every finite-length module is finitely generated.
Conversely, any charge-graded $F_N$-module is the
directed union of its finitely generated graded
submodules, which have finite length by
Theorem~\ref{thm:irred}.
The functor $F_0$ preserves these directed unions.
If its image is finitely generated, it is already
the image of one such finite-length submodule.
This proves the assertion about finite generation.

The projective covers and their radical filtrations
follow from Remark~\ref{rm:loewy} and exactness.
The charge shifts are identified with spectral flow
by Theorem~\ref{thm:sflow}.
A nonzero action of a generator of charge $1$
gives a nonsplit extension linking the trivial
modules in two adjacent charges.
Hence all the simple objects lie in one block.
\end{proof}

\begin{theorem}[Vacuum extension algebra]\label{rem:ext}
For $N\geq2$, there is an isomorphism of graded
associative algebras
\[
\Ext^\bullet_{\Bcat_N}(B_N,B_N)
\cong
\mathbb C[x_1,\ldots,x_N,y_1,\ldots,y_N]^{\mathbb C^*}
\cong
\mathbb C[z_{ij}\mid1\leq i,j\leq N]/I_2.
\]
Here the variables $x_i,y_i$ have cohomological
degree one and charges $1,-1$, respectively;
$z_{ij}=x_iy_j$ has cohomological degree two,
and $I_2$ is the ideal generated by the
$2\times2$ minors of $(z_{ij})$.
In particular, the odd-degree self-extension
groups vanish.
\end{theorem}

\begin{proof}
Let
\[
Z(\mathfrak g_N)=\Lambda(E),
\qquad
E=\mathbb C^N_{+1}\oplus\mathbb C^N_{-1}.
\]
The charge-graded Koszul resolution of the trivial
module has terms
\[
\Lambda(E)\otimes\operatorname{Sym}^j(E)
\]
in homological degree $j$.
Here $\operatorname{Sym}$ denotes the ordinary
symmetric algebra, with
$\operatorname{Sym}^j(E)$ placed in parity $j\pmod2$.
The differential is the sum of right multiplication
by an exterior generator on the first factor
followed by differentiation in the corresponding
symmetric variable on the second.

Applying charge- and parity-preserving
$\operatorname{Hom}_{A_N}(-,\mathbb C)$ kills the
differential and gives the charge-zero part of
$\operatorname{Sym}^j(E^*)$.
Charge zero already forces even parity in this
convention.
The Koszul product identifies the Yoneda algebra with
$
\operatorname{Sym}(E^*)^{\mathbb C^*}.
$
Choose $x_i$ dual to the charge $-1$ generators
and $y_i$ dual to the charge $1$ generators,
so dualization gives the stated charges.

Every invariant monomial is a product of the $x_iy_j$.
The relations among these generators are
\[
z_{ij}z_{kl}=z_{il}z_{kj}.
\]
These exchanges identify precisely the products
with the same row and column multiplicities.
They generate $I_2$, proving the second isomorphism.
The equivalence in Theorem~\ref{thm:projectives}
gives the assertion for $B_N$.
\end{proof}

\begin{proof}[Proof of Theorem~\ref{thm:Loewy}]
By Theorem~\ref{thm:projectives}, the projective
cover of $V=B_2$ in $\Bcat_2$ is
\[
P_0=F_0\mathcal I_2(Z(\mathfrak g_2)).
\]
The exact equivalence preserves radical filtrations,
and Theorem~\ref{thm:sflow} identifies charge shifts
with spectral flow.
Thus Remark~\ref{rm:loewy} gives the layers
\[
\begin{gathered}
V^{(0)},\\
(V^{(1)})^{\oplus2}\oplus(V^{(-1)})^{\oplus2},\\
V^{(2)}\oplus(V^{(0)})^{\oplus4}\oplus V^{(-2)},\\
(V^{(1)})^{\oplus2}\oplus(V^{(-1)})^{\oplus2},\\
V^{(0)}.
\end{gathered}
\]
The arrows are transported from multiplication
by the four exterior generators, proving the
asserted diagram in $\Bcat_2$.

For the Hamiltonian assertions, the minimum
$F_2$ energy in charge $\ell$ is
\[
e_2(\ell)
=
\sum_{j=0}^{|\ell|-1}\left(\frac32+j\right)
=
\frac{\ell^2}{2}+|\ell|,
\]
attained by the lowest distinct creators of one sign.
The negative-lattice contribution is $-\ell^2/2$.
Hence the corresponding vector in $V$ has minimum
energy $|\ell|$ and $J_0$-charge $-\ell$.

On $V^{(r)}$, inverse pullback gives
\[
J_0^{(r)}=J_0-r,
\qquad
H^{(r)}=H+rJ_0.
\]
The latter operator is compatible with the
spectral-flow module action.
Consequently, the resulting charge is
$q=-(\ell+r)$, and the attained minimum energy
at charge $q$ is
\[
|\ell|-r\ell
=
|q+r|+r(q+r).
\]
The free-field realization gives finite-dimensional
joint charge--energy spaces and a lower bound at
each fixed charge.
Since the spectral-flow modules are simple,
they are finitely generated.
Thus $(V^{(r)},H^{(r)})$ belongs to
$V$-\textup{wtmod}.

For $r=0$, only finitely many charges occur at
each energy, so $(V^{(0)},H^{(0)})$ belongs to
$V$-\textup{wtmod}$_{KL}$.
For $r=1$, all energies are nonnegative, and at
fixed energy $\Delta$ the charges satisfy
$q\leq\Delta/2-1$.
However, every $q\leq-1$ occurs at energy zero.
Thus $(V^{(1)},H^{(1)})$ belongs to
$V$-\textup{wtmod}$_{\mathcal O}$ but not to
$V$-\textup{wtmod}$_{KL}$.
For $r=-1$, all energies are nonnegative, but every
$q\geq1$ occurs at energy zero.
Hence $(V^{(-1)},H^{(-1)})$ belongs to
$V$-\textup{wtmod}$_{\geq}$ but not to
$V$-\textup{wtmod}$_{\mathcal O}$.
For $r=2$, the minimum is $q+2$ when $q\leq-2$;
for $r=-2$, it is $-q+2$ when $q\geq2$.
Both are unbounded below.

Finally, the horizontal zero modes on $V^{(1)}$
and $V^{(-1)}$ are respectively
\[
(e_1,f_{-1},h_0-2),
\qquad
(e_{-1},f_1,h_0+2).
\]
Their energy-zero spaces are generated by the
original vacuum as a highest-weight module of
weight $-\alpha$ and a lowest-weight module of
weight $\alpha$, respectively.
Simplicity therefore gives
\[
V^{(1)}\cong D^+_{-\alpha},
\qquad
V^{(-1)}\cong D^-_{\alpha}.
\]
\end{proof}

\subsection*{Acknowledgements}
 T. Creutzig is supported by DFG project Projektnummer 551865932.
 X.Dai acknowledges support from JSPS KAKENHI Grant Numbers 21H04993.
 B. Song is supported by NSFC No.12671042.
We thank Shigenori Nakatsuka and Peter Fiebig for their helpful remarks.

\noindent\textbf{AI declaration.}
The first version of the manuscript was written without AI assistance. OpenAI GPT-6.1 was later used for proofreading and checking mathematical arguments. 
The authors reviewed its suggestions and are solely responsible for the manuscript,
including all mathematical results and any remaining errors.

\end{document}